\documentclass[11pt]{amsart}
\usepackage[margin=1.05in]{geometry}
\usepackage{xr-hyper,needspace}
\usepackage[T1]{fontenc}
\usepackage{lmodern,amsmath,amssymb,amsthm,mathtools,booktabs,array,longtable,mathrsfs,microtype,xcolor,hyperref,enumitem}
\hypersetup{colorlinks=true,linkcolor=blue!55!black,citecolor=blue!55!black,urlcolor=blue!55!black,pdftitle={Frobenius multiplicities and F-signatures of Grassmannians of two-planes}}
\newtheorem{theorem}{Theorem}[section]
\newtheorem{proposition}[theorem]{Proposition}
\newtheorem{lemma}[theorem]{Lemma}
\newtheorem{corollary}[theorem]{Corollary}
\theoremstyle{definition}

\newtheorem{example}[theorem]{Example}
\newtheorem{remark}[theorem]{Remark}
\newcommand{\Gr}{\operatorname{Gr}}
\newcommand{\ch}{\operatorname{ch}}
\newcommand{\Ind}{\operatorname{Ind}}

\newcommand{\rk}{\operatorname{rank}}
\newcommand{\rank}{\operatorname{rank}}

\newcommand{\SL}{\operatorname{SL}}
\newcommand{\Tor}{\operatorname{Tor}}
\newcommand{\Hom}{\operatorname{Hom}}

\newcommand{\Sym}{\operatorname{Sym}}

\newcommand{\adj}{\operatorname{adj}}

\title[Frobenius multiplicities of Grassmannians of two-planes]{Frobenius multiplicities and $F$-signatures of Grassmannians of two-planes}
\date{September 29, 2026}
\begin{document}
\author[Lyle]{Justin Lyle}
\email[Justin Lyle]{jll0107@auburn.edu}
\urladdr{https://jlyle42.github.io/justinlyle/}
\address{Department of Mathematics and Statistics \\ 305 W Samford Avenue \\
Auburn University, AL 36849}



\begin{abstract}
Raedschelders, \v Spenko and Van den Bergh proved that the Pl\"ucker
coordinate ring $C_n$ of the Grassmannian $\Gr(2,n)$ has finite
Frobenius representation type when the characteristic $p$ satisfies $p\ge\max\{n-2,3\}$, and they determined the
indecomposable summands of its Frobenius pushforwards up to nonzero
multiplicity. We give nonrecursive finite formulas for every entry of
the Frobenius multiplicity matrix and prove that each entry is a
polynomial in $p$ of degree at most $2n-3$.
Consequently the $F$-signature of $C_n$, and the asymptotic proportion of
every indecomposable summand, is a single rational function of $p$
throughout the characteristic range. Its limit as $p\to\infty$ is the
leading coefficient of the first free multiplicity, which we evaluate in
closed form through Bernoulli numbers. For
$\Gr(2,4)$ we find $s_p(C_4)=(13p^2+8)/(5(3p^2+2))$, and we calculate the analogous rational
functions for $\Gr(2,5)$ and $\Gr(2,6)$ with computer assistance.
We also construct the full characteristic polynomial from the
Riemann--Roch weights and two residual matrices whose entries have
degree at most $n-1$ in $p$. In the sequel, we use this work to show rank-two determinantal rings have FFRT, and to calculate their $F$-signatures. 
\end{abstract}
\maketitle

\section{Introduction}
Let $k$ be an algebraically closed field of characteristic $p>0$ and
$R$ a positively graded $k$-algebra with $R_0=k$. Write $F^e_*R$ for
$R$ regarded as a module over itself through the $e$th Frobenius map and let $a_e(R)$ denote the maximal rank of an $R$-free summand of $F^e_*R$.
The ring $R$ has \emph{finite Frobenius representation type} (FFRT) if the
indecomposable summands of all the $F^e_*R$ lie in a finite set of
isomorphism classes, up to graded shift \cite{SVdB}. Such rings have
especially tractable Frobenius behavior. Indeed, once the multiplicity of each
summand in the pushforward of each other is known, the \emph{$F$-signature}
\[
 s(R):=\lim_{e\to\infty}\frac{a_e(R)}{p^{e\dim R}}
\]
\cite{HL,Tucker}, and the asymptotic proportions of the other summands
follow from a Perron eigenvector calculation \cite{SeibertFFRT,Yao}. In practice,
however, determining the multiplicities is considerably harder than identifying the summands themselves, and this information is needed for the $F$-signature calculation.

The Pl\"ucker coordinate ring
\[
 C_n=k[\Gr(2,n)]=\Sym(k^n\otimes k^2)^{\SL_2},\qquad \dim C_n=2n-3,
\]
is the fundamental example. Raedschelders, \v Spenko and Van den Bergh
prove that it has FFRT for $p\ge\max\{n-2,3\}$
\cite{RSVdB1,RSVdB2}; as they note, this is the first non-trivial
example of FFRT for invariants of a group that is not linearly
reductive. They identify a family $\mathscr G_n$ of
$(n-2)+(n-3)^2$ indecomposable modules, which consists of the modules of
covariants $S\{j\}$, $0\le j\le n-3$ and certain syzygy modules
$K\{j,k\}$, $1\le j,k\le n-3$, and they prove that $F_*C_n$ contains exactly
the $S\{j\}$ and the $K\{j,j\}$, while every member of $\mathscr G_n$
occurs in $F^e_*C_n$ for $e\ge2$ \cite[Theorems~1.5--1.6]{RSVdB2}.
Their methods determine which summands occur, but outside of the first pushforward of $R$ \cite[Theorem 15.1]{RSVdB1}, they do not determine how often.  Our first main result completes this picture by providing these missing multiplicities.

For each admissible prime $p$, let $A_n(p)$ be the multiplicity
matrix on this family, so the entry $A_n(p)_{M',M}$ counts copies of $M'$
in $F_*M$ with all grading shifts counted together. In particular, rows index
the summands being counted, and columns index the modules being
pushed forward.

\begin{theorem}\label{thm:A}
For any fixed $n \ge 4$, every entry of $A_n(p)$ agrees with a polynomial in
$\mathbb Q[p]$ of degree at most $2n-3$ for every prime
$p\ge\max\{n-2,3\}$.
\end{theorem}

The entries are given by nonrecursive finite formulas whose coefficients
and summation ranges are independent of $p$. Section~\ref{gr:section}
and Appendix~\ref{gr:appendix} reconstruct the matrix from three counting
functions, while Section~\ref{gr:poly} evaluates those functions by polynomial
formulas and verifies the characteristic range and degree bound.
This polynomial dependence across characteristics does not follow from FFRT alone. For instance, the Veronese subring
$R=k[x,y]^{(3)}$ has FFRT for $p\ne3$, but the free rank of $F_*R$ is
$(p^2+2)/3$ when $p\equiv1\pmod3$ and $(p^2-1)/3$ when $p\equiv2\pmod3$,
so no single polynomial in $p$ gives it at every prime.

For the quadric cone $C_4$,
whose family consists of $C_4$ and its two spinor modules $S\{1\}$ and
$K\{1,1\}$, we find, up to graded shifts, that
\[
 F_*C_4\cong C_4^{\oplus p(13p^4+2)/15}\oplus
       S\{1\}^{\oplus p(p^4-1)/30}\oplus K\{1,1\}^{\oplus p(p^4-1)/30},
\]
so that for instance
$F_*C_4\cong C_4^{\oplus211}\oplus S\{1\}^{\oplus8}\oplus K\{1,1\}^{\oplus8}$
when $p=3$. The full matrix $A_4(p)$ is given in
Section~\ref{gr:example}. Since $C_4$ is the homogeneous coordinate ring of the smooth quadric
$\Gr(2,4)\subset\mathbb P^5$, these multiplicities are also determined by
Achinger's description of the higher Frobenius push-forwards of line and
spinor bundles on quadrics \cite{Achinger}, which extends Langer's
description of the first push-forward \cite{Langer}. Their expressions are in terms of
dimensions of graded pieces of explicit Artinian modules, and our
formulas agree with them.

The formulas underlying Theorem~\ref{thm:A} determine all Frobenius
multiplicities of $C_n$, with internal shifts counted together, and thus allow us to calculate $F$-signatures. While a closed form for $s_p(C_n)$ for arbitrary $n$ is not given in our work, we do give the explicit form for the limit $F$-signature $\lim_{p \to \infty} s_p(C_n)$. To present this result, we let $B_{2m}$ denote the Bernoulli numbers.

\begin{theorem}\label{thm:B}
Suppose $n\ge4$ and $p\ge\max\{n-2,3\}$.
\begin{enumerate}
\item For each $M\in\mathscr G_n$ there is a rational function
$v_M\in\mathbb Q(x)$ such that the multiplicity of $M$ in $F^e_*C_n$ is
$v_M(p)\,p^{e(2n-3)}+o(p^{e(2n-3)})$ as $e\to\infty$. In particular
$s_p(C_n)=v_{C_n}(p)$ is a single rational function of $p$ throughout
this range.
\item The $F$-signature is determined by a linear system in
$\binom{n-2}{2}$ unknowns, obtained from $\lfloor n/2\rfloor$
Riemann--Roch eigenrows of the multiplicity matrix.
\item As $p\to\infty$ through primes, we have
\[
 \lim_{p\to\infty}s_p(C_n)=
 -\frac{2}{(2n-3)!}\sum_{k=0}^{n-1}(-1)^k\binom{2n}{k}(n-k)^{2n-3}
 -(2^n-2)\frac{2^n(2^{2n-2}-1)(-1)^nB_{2n-2}}{(2n-2)!}.
\]
\end{enumerate}
\end{theorem}

Part~(3) is the leading coefficient of the free multiplicity
$A_n(p)_{C_n,C_n}$, but this identification is not formal, as a polynomial matrix can have a normalized Perron vector whose limit is not read off from its leading terms. 
Since $C_n$ is Gorenstein, we obtain this identification from the
socle formula for splitting numbers and Smirnov's uniform
Hilbert--Kunz estimate; see Section~\ref{sig:section}. The key point is to show that $\displaystyle \lim_{p \to \infty} s_p(C_n)=\lim_{p \to \infty} \dfrac{a_1(C_n)}{p^{\dim C_n}}$.

For $n=4,\dots,8$ the limits are
\[
 \frac{13}{15},\quad\frac{97}{126},\quad\frac{15067}{22680},\quad
 \frac{2983}{5280},\quad\frac{11677453}{24324300}.
\]
For $C_4$, we calculate $a_e(C_4)$ explicitly, and we provide closed form formulae for $s_p(C_4)$ for every odd prime:
\begin{theorem}\label{thm:C}
For every prime $p\ge3$, we have
\[
 s_p(C_4)=\frac{13p^2+8}{5(3p^2+2)},
\]
and
\[
 a_e(C_4)=\frac{13p^2+8}{5(3p^2+2)}\,p^{5e}
 +\frac{2(p^2+1)}{5(3p^2+2)}\left(\frac{p(p^2+2)}3\right)^{e}.
\]
The $F$-signatures of $C_5$ and $C_6$ are rational functions of $p$ of
degrees $9/9$ and $30/30$, given in Section~\ref{ex:examples}. For
instance
\[
 s_p(C_6)=\frac{15067}{22680}-\frac{899}{48600\,p^4}+O(p^{-5}).
\]
\end{theorem}

The limit $13/15$ of $s_p(C_4)$ agrees with Shideler's formula for the
limit $F$-signature of a sum of six squares \cite{Shideler}; see also
\cite{CSTZ}.
Theorem~6.4 of \cite{CSTZ}, applied to six squares, already gives a
free-rank sequence of the form $s_p(C_4)p^{5e}+(1-s_p(C_4))B^e$, with integers
$B,C$ such that $s_p(C_4)=-C/(p^5-B)$. There $B$ is the coefficient of
$\lambda_{(p-1)/2}$ in $\lambda_{(p-1)/2}^5$ in the Han--Monsky
representation ring, and the authors note that $B$ and $C$ are hard to
determine explicitly when $B>1$. Theorem~\ref{thm:C} identifies
\[
 B=\frac{p(p^2+2)}3,\qquad C=-\frac{p(p^2-1)(13p^2+8)}{15}.
\]
The signature $s_p(C_4)$ is also
determined in principle by Achinger's description, through the Perron
vector of the resulting $3\times3$ matrix, but we have not found this
closed expression for fixed $p$ in the literature. For $n\ge5$ the ring $C_n$ is not a hypersurface, and to our
knowledge no $F$-signature of $C_n$ with $n\ge5$ was previously known.

\subsection*{The method}\label{subseq:method}
As noted above, Theorem~\ref{thm:A} does not follow from FFRT alone.
Its proof takes as input the structural results of \cite{RSVdB2}, in particular, the
Frobenius-kernel invariants of covariant modules
\cite[Proposition~5.1, Lemma~5.4]{RSVdB2}, the decomposition theorem
\cite[Theorem~4.3, Proposition~9.1]{RSVdB2}, the single surviving class
of their spectral sequence \cite[Lemma~7.8]{RSVdB2}, and induction from
Borel invariants \cite[Proposition~8.2]{RSVdB2}. It adds two things: an
exact count of each multiplicity, and a proof that these counts are
fixed finite sums of polynomials in $p$. Write $Q$ for the
truncated polynomial ring $\Sym(k^n\otimes k^2)/(x_i^p,y_i^p)$. There are
three mechanisms. Here a \emph{source} is the module being pushed
forward, and a \emph{target} is a summand whose multiplicity is counted.
\begin{enumerate}
\item For a covariant source $S\{j\}$, the covariant targets are read
off from the tilting character of $T(j)\otimes Q$. This character
determines the relevant invariants since the module is tilting. The
syzygy targets are fusion counts in the level-$(p-2)$ fusion ring of
$\SL_2$.
\item For a syzygy source $K\{j,k\}$, the syzygy targets are dimension
counts on the spectral sequence of \cite{RSVdB2}. After fixing a tuple and
a torus weight, its $E_3$ page is concentrated in one column with at most
one surviving class, so no rank of a map is ever needed.
\item The covariant targets of a syzygy source require a different
calculation. The natural module $K_{jk}\otimes_SQ$ is not known to be tilting,
so we cannot determine the count from its character. Instead,
\cite[Proposition~8.2]{RSVdB2} expresses the relevant invariants as
being induced from Borel invariants. We reconstruct the character of those
Borel invariants from Tate cohomology, by Jordan-block arithmetic over
the distribution algebra $k[f]/(f^p)$ of the first Frobenius kernel of
the unipotent radical.
\end{enumerate}
In each case the multiplicities are finite sums of three counting
functions (a weight multiplicity in $Q$, a fusion count, and a bounded
composition count) evaluated at arguments affine in $p$. The last step
is to show that these are polynomials in $p$ throughout the range. For
the fusion counts this uses a Verlinde-type evaluation in terms of sums
of powers of cosecants, and for the others the vanishing of binomial
coefficients $\binom{x+r}{r}$ at $-r\le x\le-1$.

Part~(2) of Theorem~\ref{thm:B} comes from Riemann--Roch on the punctured
cone over $\Gr(2,n)$. Its rational Chow ring is $\mathbb Q[u]/(u^{\lfloor
n/2\rfloor})$, and Frobenius acts on the component of codimension $2i$
by $p^{2n-3-2i}$, giving $\lfloor n/2\rfloor$ eigenrows of the
multiplicity matrix that are independent of $p$.

The same geometry leads to the full characteristic polynomial
(Theorem~\ref{gr:characteristic-polynomial}), with the Riemann--Roch factor
being complemented by two residual determinants, obtained by splitting
formal resolution relations under duality. Both residual matrices are
constructed from the bounded-composition counts alone and have entries
of degree at most $n-1$ in $p$. Thus their construction avoids first
assembling the full multiplicity matrix, while retaining the entire
characteristic polynomial.

The general theorems are proved by hand, while the rational functions for
$C_5$ and $C_6$ are computer-assisted, using an accompanying program for some necessary linear algebra calculations, with no interpolation. The program is available as an ancillary file.

In the sequel we apply these results to the generic determinantal rings
$k[X_{a\times b}]/I_3(X)$, which are weight-zero parts of the rings $C_n$
for a one-dimensional torus. In the corresponding characteristic
range, FFRT descends from the Grassmannian ring by
\cite[Proposition~3.1.6]{SVdB}. The sequel then determines the minimal
summand families and their multiplicities, and computes $F$-signatures,
including that of the $3\times3$ generic determinantal hypersurface.

We now describe the structure of this work.
Section~\ref{sec:ffrt} recalls the FFRT formalism.
Section~\ref{gr:section} derives the multiplicity formulas, with
$\Gr(2,4)$ as a running example, and Section~\ref{gr:poly} evaluates the
counting functions by polynomial formulas, completing the proof of
Theorem~\ref{thm:A}. Section~\ref{sig:section} proves the
characteristic-limit formula in Theorem~\ref{thm:B}(3), while
Section~\ref{gr:rr-reduction} gives the Riemann--Roch reduction of
Theorem~\ref{thm:B}(2) and the characteristic-polynomial construction
of Theorem~\ref{gr:characteristic-polynomial}. Explicit signatures are given in Section~\ref{ex:examples}. Finally, the boundary
specialization of the reconstruction formulas appears in
Appendix~\ref{gr:appendix}.

\subsection*{Preliminaries and conventions}\label{sec:preliminaries}
All modules are finitely generated unless otherwise stated. The grading
on $F_*^eM$ is divided by $p^e$. 
Frobenius summands are considered up
to internal shift, and a multiplicity counts all shifts of that type.
We retain shifts in the resolutions defining the modules. The number
$a_e(R)$ counts all graded free summands of $F_*^eR$, with their internal
shifts counted together, while $\rk$ denotes rank over the fraction field. Throughout, we let $n\ge4$, let $p\ge\max\{n-2,3\}$, and let
\[
 C_n=k[\operatorname{Gr}(2,n)]
     =\operatorname{Sym}(k^n\otimes k^2)^{\operatorname{SL}_2}
\]
be the Pl\"ucker coordinate ring, which has dimension $d_n=2n-3$. The coordinate ring $C_n$
has its standard Pl\"ucker grading, while in the invariant presentation below
a Pl\"ucker generator has polynomial degree two. These gradings differ
only by this normalization. Displayed shifts in the equivariant
constructions and their invariant modules use the polynomial grading.

For the representation-theoretic calculations, put $V=k^2$, $F=k^n$,
$G=\SL(V)$, and $S=\Sym(F\otimes V)=k[x_i,y_i]_{i=1}^n$.
Let $B$ be the lower triangular Borel, $\mathbb T$ its diagonal torus,
and $U$ its unipotent radical. 
A subscript $1$ denotes the kernel of
the first Frobenius morphism, as in $G_1,B_1,U_1$.
We identify torus weights with integer multiples of the fundamental
weight $\omega$, with $x_i$ of weight $+\omega$ and $y_i$ of weight
$-\omega$. A one-dimensional $B$-module of weight $j\omega$ is denoted
by $(j\omega)$, while $\Ind_B^G$ denotes rational induction.

The symbols $L(j)$, $\nabla(j)$ and $T(j)$ denote the simple, induced
and indecomposable tilting $G$-modules of highest weight $j\omega$.
We note that for $0\le j\le p-1$, all three agree with $\Sym^jV$.
The character $\ch M$ records torus-weight dimensions, and
$\chi_j=\ch\nabla(j)=z^j+z^{j-2}+\cdots+z^{-j}$.
The superscript $[1]$ denotes Frobenius twist. 
We use ordinary duals: $M^\vee=\Hom_A(M,A)$ for a module over a ring
$A$, and $W^*=\Hom_k(W,k)$ for a representation.

Write
\[
 S_+=S/(y_1,\ldots,y_n),\qquad
 Q=S/S^p_{>0}S=S/(x_1^p,y_1^p,\ldots,x_n^p,y_n^p).
\]
Here $S^p_{>0}$ is the positive-degree ideal of the subring of
$p$th powers. For $0\le j\le n-3$, the covariant modules are
$S\{j\}=(T(j)\otimes_kS)^G$, with $S\{0\}=C_n=S^G$.
For $1\le j\le n-3$, put
$M_j=\Ind_B^G((j\omega)\otimes_kS_+)$.
In our characteristic range, the equivariant graded free resolution
\cite[Remark~3.4 and (3.5)]{RSVdB2} is
\begin{equation}\label{eq:defining-resolution}
 0\longrightarrow E_{j,n-1}\longrightarrow\cdots
 \longrightarrow E_{j,0}\longrightarrow M_j\longrightarrow0,
\end{equation}
where
\[
 E_{j,a}=
 \begin{cases}
  T(j-a)\otimes_k\bigwedge^aF\otimes_kS(-a),&0\le a\le j,\\
  T(a-j-1)\otimes_k\bigwedge^{a+1}F\otimes_kS(-a-1),&j<a\le n-1.
 \end{cases}
\]
Using this resolution, define
\[
 K_{jk}=\Omega^{k+1}M_j
       =\ker(E_{j,k}\longrightarrow E_{j,k-1}) \qquad \text{ and } \qquad
 K\{j,k\}=(K_{jk})^G,\quad \text{ for } 1\le j,k\le n-3.
\]
Thus $K_{jk}$ is an equivariant $S$-module and $K\{j,k\}$ is a
$C_n$-module. The choice of resolution is part of the definition.
We write $H^i$ for rational group-scheme cohomology and $\widehat H^i$
for Tate cohomology. Further notation from \cite{RSVdB2} is introduced
where it is used.

\section{The FFRT matrix and its Perron eigenspace}\label{sec:ffrt}
We recall the classical FFRT method for encoding Frobenius as a matrix (see
\cite[Proposition~3.3.1]{SVdB} and \cite[Section~3]{Yao}).

\begin{theorem}\label{thm:ffrt-spectrum}
Let $R$ be a strongly F-regular F-finite positively graded domain of dimension $d$ with
perfect degree-zero field and finite F-representation type.
Let $M_0=R,M_1,\ldots,M_{h-1}$ represent the distinct indecomposable
graded summands occurring in some $F_*^eR$, up to internal shifts. Write
\[
 F_*M_j\cong\bigoplus_iM_i^{\oplus A_{ij}}\qquad \text{ and } \quad
 \rho=(\rk M_0,\ldots,\rk M_{h-1}),
\]
with graded shifts suppressed. Then $\rho A=p^d\rho$, the eigenvalue $p^d$
is simple, and its positive right eigenvector $v$, normalized by
$\rho v=1$, satisfies
\[
 \lim_{e\to\infty}p^{-de}A^e=v\rho \qquad \text{ with } \qquad s(R)=v_0.
\]
In particular, with $\chi_A(t)=\det(tI-A)$, we have
\begin{equation}\label{eq:signature-adjugate}
 s(R)=\frac{\adj(p^dI-A)_{00}}{\chi_A'(p^d)}.
\end{equation}
\end{theorem}
\begin{proof}
The family is Frobenius-closed because each $M_j$ occurs in an iterate
of $R$. Smith--Van den Bergh's result and its proof show that $A$ is
primitive and give the Perron convergence
\cite[Proposition~3.3.1]{SVdB}. Taking ranks gives
$\rho A=p^d\rho$, so this is the Perron eigenvalue and the stated
normalization identifies the limit as $v\rho$.
Since $(A^e)_{00}$ is the free multiplicity of $F_*^eR$ and $\rho_0=1$,
its normalized limit is $v_0=s(R)$. Finally, the adjugate at the simple
root $p^d$ is
$\adj(p^dI-A)=\chi_A'(p^d)v\rho$, so its $(0,0)$-entry gives the formula.
\end{proof}

\begin{corollary}\label{cor:ffrt-rational}
Suppose the family and rank vector are fixed and $A(p)$ is the
evaluation of a matrix over $\mathbb Q[x]$ at every prime in an
unbounded range. Then $s(R_p)$ is a single rational function of $p$
on that range, with a finite rational limit as $p\to\infty$.
Every matrix entry has degree at most $d$.
\end{corollary}
\begin{proof}
Use \eqref{eq:signature-adjugate} and $0<s(R_p)\le1$. Nonzero
entries of $A(x)$ have positive leading coefficient, so the rank
identity $\rho A(x)=x^d\rho$ bounds their degrees.
\end{proof}

The Pl\"ucker rings considered below are Gorenstein. In this case we
use $M^\vee=\Hom_R(M,R)$ as noted in Section \ref{sec:preliminaries}. Finite Frobenius duality identifies
$(F_*M)^\vee$ with $F_*(M^\vee)$ up to an internal shift, so
in particular, duality permutes the Frobenius summand types, commutes with $A$,
and fixes its normalized Perron vector. Moreover, free summands are invariant under this duality.
Let $H$ record total multiplicities on duality orbits. That is, for each source
representative, sum the entries of $A$ over a target orbit. Its
normalized Perron vector records the total limiting multiplicity on
each orbit, and
\begin{equation}\label{eq:duality-quotient}
 s(R)=\frac{\adj(p^dI-H)_{00}}{\chi_H'(p^d)}.
\end{equation}
For an arbitrary source, its multiplicity column need not be constant
on target duality orbits. The multiplicities in $F_*^eR$ are invariant
under duality, since the source $R$ is self-dual, and the limiting
vector is invariant as well. Polynomial dependence alone does not
identify the characteristic limit with the leading coefficient of
$A_{00}$. That identification follows from the Gorenstein socle formula
and uniform Hilbert--Kunz convergence in Section~\ref{sig:section}.

\section{Frobenius pushforward on the Pl\"ucker ring}
\label{gr:section}

Throughout this section internal grading shifts are counted but suppressed.

As identified in \cite{RSVdB1,RSVdB2}, the family of Frobenius summand types over $C_n$ is
\[
 \mathscr G_n=\{S\{j\}:0\le j\le n-3\}
     \cup\{K\{j,k\}:1\le j,k\le n-3\} \qquad \text{with} \qquad S\{0\}=C_n.
\]
By \cite[Theorems~1.5--1.6]{RSVdB2}, the types in $F_*C_n$ are
the $S\{j\}$ and $K\{j,j\}$, while every member of $\mathscr G_n$
occurs in $F_*^eC_n$ for every $e\ge2$. There are therefore $2n-5$
types at the first iterate and $(n-2)+(n-3)^2$ thereafter.
Recall $A_n(p)_{M',M}$ denotes the multiplicity of the target $M'$ in the
Frobenius pushforward of the source $M$.

The ranks of $K\{j,k\}$ equal the $S$-ranks of $K_{jk}$. Indeed, on the
full-rank matrix locus, taking invariants is descent along an
$\SL_2$-torsor over the punctured cone. Taking alternating ranks in
the tail of \eqref{eq:defining-resolution} ending in $K_{jk}$ gives
\begin{equation}\label{det:ranks}
 \rk S\{j\}=j+1 \qquad \text{and} \qquad
 \rk K\{j,k\}=
  \begin{cases}
   (j-k+1)\binom{n-2}{k}+(j-k)\binom{n-2}{k-1},&k\le j\\
   (k-j+1)\binom{n-2}{k}+(k-j)\binom{n-2}{k+1},&k\ge j.
  \end{cases}
\end{equation}
Note the two expressions for $\rank K\{j,k\}$ agree when $j=k$. 
The goal of this section will be to compute $A_n(p)$.
\subsection{Construction of the matrix}\label{gr:mechanisms}
We construct $A_n(p)$ one column at a time, using the ordered family
\[
 S\{0\},\ldots,S\{n-3\},\qquad
 K\{1,1\},\ldots,K\{1,n-3\},K\{2,1\},\ldots,K\{n-3,n-3\},
\]
with the pairs ordered lexicographically as above. The following recipe determines
every entry, while the subsequent subsections justify its steps.

\begin{enumerate}
\item \emph{Evaluate the counting functions.}
The construction uses three functions: $W_n(w)$ counts weights in the
truncated polynomial ring $Q$, $H_{n,q}$ is a fusion multiplicity, and
$E_{j;q,e}$ counts bounded compositions of $(q+e)p-j-2e$ into $n$
parts. Their definitions are given in \eqref{gr:weight}, \eqref{gr:H} and
\eqref{gr:E}. To compute symbolically in $p$, use the binomial formula
\eqref{gr:C} for $W$ and $E$, and the finite Bernoulli-number formula
\eqref{gr:H} for $H$, with coefficients given by \eqref{gr:b} and
\eqref{gr:cosecant-coefficients}. All ranges depend only on $n$.

\item \emph{Fill a covariant-source column $S\{j\}$.}
Form the Weyl coefficients $a^{(j)}_\lambda$ from $W_n$ using
\eqref{gr:weyl-coefficients}. Equation~\eqref{gr:SC} gives the free
row, and the finite alternating sum \eqref{gr:SS} gives every other
covariant row. Equation~\eqref{gr:SK} places $H_{n,j}$ in every
diagonal syzygy row $K\{r,r\}$ and zero in every off-diagonal row.
This completes the column.

\item \emph{Fill a syzygy-source column $K\{j,k\}$ with $k\le j$.}
First compute the generic term $B_{j,k}$ from the fusion counts using
\eqref{gr:B}. Combine it with the $E$-counts in \eqref{gr:T} and
\eqref{gr:O} to obtain the Tate coefficients $\mathcal T$ and the
ordinary-cohomology coefficients $\mathcal O$. The syzygy rows are
already determined:
\[
 A_{K\{r,h\},K\{j,k\}}=\mathcal O_{j,k;h,r}.
\]
To obtain the covariant rows, form $P$ from $\mathcal T$ and $W_n$
using \eqref{gr:P}, and then form $Z$ using the finite sum
\eqref{gr:Z}. The difference $Z_{j,k;l}-Z_{j,k;-l-2}$ is the
coefficient after Euler induction, before removing the syzygy
summands. Subtract their resolution characters, with multiplicities
$\mathcal O$, exactly as in \eqref{gr:KS}. The result is
$A_{S\{l\},K\{j,k\}}$ for each $0\le l\le n-3$.

\item \emph{Fill the remaining columns by duality.}
If $k>j$, the dual source is $K\{n-2-j,n-2-k\}$, whose second
index is smaller than its first. Its column was computed in the
previous step. Equation~\eqref{gr:duality} then gives the desired column
by fixing the covariant rows and replacing each syzygy row
$K\{r,h\}$ by $K\{n-2-r,n-2-h\}$.
\end{enumerate}

For covariant sources, tilting characters determine the covariant
part, while fusion counts determine the syzygy part
\cite[Proposition~5.1]{RSVdB2}. For syzygy sources, the
$\mathcal O$-counts come from a spectral sequence with a single
surviving class. The covariant part requires the additional
reconstruction through $P$ and $Z$, as the character of
$K_{jk}\otimes_SQ$ alone is not known to determine its decomposition.
Instead, Borel invariants and Euler induction recover a character
from which the known syzygy contributions can be subtracted
(\S\ref{gr:ks}). Example~\ref{gr:example} serves as an illustration of the matrix construction, and carries out both kinds of
column calculations for $n=4$.

\subsection{Covariant sources}\label{gr:cov}
Fix a source $S\{j\}$. We first count weights in $T(j)\otimes Q$,
convert them into Weyl-character coefficients, and use finite alternating
sums to recover the covariant multiplicities. The two-term tilting
characters below explain these sums. A cohomological correction is
needed at the free target; fusion counts give the diagonal syzygy targets.
Write $C_{r,p}(M)=[x^M](1+x+\cdots+x^{p-1})^r$ and
$W_n(w)=C_{2n,p}(n(p-1)+w)$ for the multiplicity of the
$\SL_2$-weight $w$ in $Q=S/S^p_{>0}S$. Reflecting the $y$-exponents
converts the weight into a total degree. Let
$H_{n,q}(p)=[V_q](\sum_{i=0}^{p-2}V_i)^n$ be the fusion count in the
level-$(p-2)$ $\SL_2$ fusion ring, with the labels $p-1$ and $-1$
interpreted as zero. For a covariant source,
\cite[Proposition~5.1 and Lemma~5.4]{RSVdB2} reduce the calculation to
the tilting character of $T(j)\otimes Q$, which determines the
tilt-free part because $T(j)\otimes Q$ is tilting, and to fusion counts
for the higher $G_1$-cohomology, which give the $K$-targets.

For $0\le j\le n-3$, all highest weights in
$T(j)\otimes Q$ are at most $n(p-1)+j\le np-3$.
First suppose $p\ge n-1$. In the SL2 tilting tensor-product formula
(see \cite{DotyHenke}), we have
$T(c)=T(c_1)\otimes T(d)^{[1]}$, with
\(p-1\le c_1\le2p-2\). The remaining digit satisfies
\(0\le d\le n-2\le p-1\), and thus $T(d)=\nabla(d)$.
The only highest weights relevant to residues $0$ and $-2$ modulo $p$
belong to
\[
 \Lambda_n=\{0,p-2\}\cup\{ip:1\le i\le n-1\}
                    \cup\{ip-2:2\le i\le n-1\}.
\]
Their characters are
\[
 \operatorname{ch}T(ip)=\chi_{ip}+\chi_{ip-2} \qquad \text{and} \qquad
 \operatorname{ch}T(ip-2)=\chi_{ip-2}+\chi_{(i-2)p},
\]
apart from the simple weights $0,p-2$.

A Weyl-character coefficient is the difference of consecutive weight
multiplicities (whose weights differ by two). For $T(j)\otimes Q$
this difference telescopes to
\begin{equation}\label{gr:weyl-coefficients}
 a^{(j)}_\lambda=W_n(\lambda-j)-W_n(\lambda+j+2),
 \qquad \lambda\in\Lambda_n.
\end{equation}
The entries of a covariant-source column are the following finite sums:
\begin{align}
 A_{S\{0\},S\{j\}}
 &=a^{(j)}_0-(2^n-2)H_{n,j},\label{gr:SC}\\
 A_{S\{l\},S\{j\}}
 &=\sum_{r=0}^{\lfloor(n-l-3)/2\rfloor}
 \bigl(a^{(j)}_{(l+2+2r)p-2}-a^{(j)}_{(l+2+2r)p}\bigr)
 \quad(1\le l\le n-3),\label{gr:SS}\\
 A_{K\{r,h\},S\{j\}}
 &=[r=h]H_{n,j}.\label{gr:SK}
\end{align}
At a Steinberg source $j=p-1$ all $H_{n,j}$ terms are zero.

To justify these formulas, let $U^{(j)}_\lambda$ be the coefficients
obtained by expressing the Weyl character in the two-term characters
above, with coefficients beyond the largest index set to zero. For
$p\ge n-1$ these are the actual tilting multiplicities. Comparing
Weyl coefficients gives
\[
 U^{(j)}_0+U^{(j)}_{2p-2}=a^{(j)}_0 \qquad \text{and} \qquad
 U^{(j)}_{ip-2}=a^{(j)}_{ip-2}-a^{(j)}_{ip}
                  +U^{(j)}_{(i+2)p-2}\quad(2\le i\le n-1).
\]
Iterating the second identity through the finite range gives exactly
the sum in \eqref{gr:SS}, which equals $U^{(j)}_{(l+2)p-2}$.
The covariant decomposition of \cite[Proposition~5.1 and Lemma~5.4]{RSVdB2}
then gives \eqref{gr:SC}--\eqref{gr:SK}.
Lemma~\ref{gr:boundary} verifies the same formulas at $p=n-2$.

\subsection{Syzygy targets}\label{gr:kk}
Recall that $K_{jk}=\Omega^{k+1}M_j$ is defined using the
$G$-equivariant resolution \eqref{eq:defining-resolution}. We also put
$L_{jk}=\Omega^{k+1}((j\omega)\otimes S_+)$, using the
$B$-equivariant twisted Koszul resolution of $(j\omega)\otimes S_+$
\cite[(3.6)]{RSVdB2}. Its term in homological degree $m$ is
\[
 ((j-m)\omega)\otimes_k\bigwedge^mF\otimes_k S(-m),
 \qquad 0\le m\le n.
\]
Assume first that $k\le j$.
After fixing a tuple and a torus weight, the spectral sequence
\cite[(6.17), Lemma~7.8]{RSVdB2} is concentrated in one column at
$E_3$. Its $E_1$ dimensions are fusion counts, and the full complex
has at most one surviving class, in position $-q$, at tuple degree
$(q+e)p-j-2e$ and row parity $r\equiv q+e\pmod2$.
Thus a brutal truncation is computed by alternating the dimensions
and correcting for the position of this class. The resulting
ordinary-cohomology multiplicities count the K-targets by
\cite[Proposition~7.6, Corollary~8.4, and Theorem~4.3]{RSVdB2}.
Appendix~\ref{gr:appendix} writes out these
alternating sums and indicators (see \eqref{gr:B}--\eqref{gr:O}) and the
resulting entries \eqref{gr:KK}.

For a source with $k>j$ we use ordinary duality. Since $C_n$ is
Gorenstein, $F_*\Hom(M,C_n)\cong\Hom(F_*M,C_n)$ up to internal shift. The
equivalence between reflexive $C_n$-modules and reflexive $(G,S)$-modules
\cite[(3.1)]{RSVdB2}
commutes with duals:
$\Hom_{C_n}(M^G,C_n)\cong\Hom_S(M,S)^G$. For the covariants this gives
$S\{j\}^\vee\cong((T(j)\otimes_kS)^\vee)^G=(T(j)^*\otimes_kS)^G\cong S\{j\}$,
since $T(j)^*=(S^jV)^*\cong D^jV=S^jV$ for $j\le n-3$
\cite[Remark~3.4]{RSVdB2}. For the syzygy modules,
\cite[Proposition~3.5(1)]{RSVdB2} gives
$K_{jk}^\vee\cong K_{n-j-2,n-k-2}(n)\otimes_k\wedge^nF^*$ as
$(G,S)$-modules, and $\wedge^nF^*$ is a trivial $G$-representation, so
taking invariants gives $K\{j,k\}^\vee\cong K\{n-2-j,n-2-k\}(n)$. 
It follows that
\begin{equation}\label{gr:duality}
 A_{Y,X}=A_{Y^{\vee},X^{\vee}},
\end{equation}
which reduces to a source whose second $K$-index is at most its first.

\subsection{Covariant targets of syzygy sources}\label{gr:ks}\label{det:borel-reconstruction}
The covariant targets of a syzygy source need the character of an
invariant module rather than of $K_{jk}\otimes_SQ$, which is not known
to be tilting, and is supplied by the following lemma. Here $z$ records the SL2 weight, while all other gradings are retained
coefficientwise.

\Needspace{8\baselineskip}
\begin{lemma}[Jordan-block reconstruction]\label{det:jordan}
Let $M$ be a finite-dimensional representation of $\mathbb T\ltimes U_1$,
where $U_1$ is the first Frobenius kernel of the unipotent radical of $B$. Write
$H_i(z)=\operatorname{ch}\widehat H^i(U_1,M)$, including its
cohomological torus weight. Then
\begin{equation}\label{det:jordan-invariants}
 \operatorname{ch}M^{B_1}=
 \left[H_0(z)+
 \frac{(1-z^2)\operatorname{ch}M+H_1(z)-H_0(z)}{1-z^{2p}}
 \right]_{p\mathbb Z},
\end{equation}
where the bracket retains torus weights divisible by $p$.
\end{lemma}
\begin{proof}
The distribution algebra $\operatorname{Dist}(U_1)=k[f]/(f^p)$ acts
on $M$, where $f$ lowers torus weight by two. Thus $M^{U_1}=\ker f$,
and we may decompose $M$ into Jordan blocks with weight-vector bases.
Both characters and Tate cohomology are additive on this decomposition,
so we first compute the expression inside the bracket for one block.

Let $J$ be a block of length $r$, with $1\le r\le p$, and let $b$
be the weight of its one-dimensional socle $\ker(f|_J)$. Its weights
are $b,b+2,\ldots,b+2r-2$, hence
\[
 \ch J=z^b(1+z^2+\cdots+z^{2r-2}) \qquad \text{and} \qquad
 \ch J^{U_1}=z^b.
\]
The two-periodic complex computing Tate cohomology has alternating
maps $f$ and $f^{p-1}$. With the cohomological torus weights included,
as in \cite[Lemma~3.7]{RSVdB2}, its characters in degrees zero and one are
\[
 \ch\widehat H^0(U_1,J)=\ch(\ker f/f^{p-1}J) \qquad \text{and} \qquad
 \ch\widehat H^1(U_1,J)=z^2\ch(\ker f^{p-1}/fJ).
\]
The factor $z^2$ in degree one compensates for the fact that $f$
lowers weight by two.

If $r<p$, then $f^{p-1}J=0$ and $\ker f^{p-1}=J$. The two Tate
characters are therefore $z^b$ and $z^{b+2r}$. Since
\[
 (1-z^2)\ch J=z^b-z^{b+2r},
\]
the numerator of the fraction in \eqref{det:jordan-invariants} is
zero for this block. Its contribution to the expression inside the
bracket is consequently $z^b=\ch J^{U_1}$.

If $r=p$, the block is free over $k[f]/(f^p)$, so both Tate groups
vanish. In this case its contribution is
\[
 \frac{(1-z^2)\ch J}{1-z^{2p}}
 =\frac{z^b(1-z^{2p})}{1-z^{2p}}=z^b,
\]
again the character of $J^{U_1}$. Summing over the blocks shows that
the expression inside the bracket is $\ch M^{U_1}$.
Finally, $B_1$ is the semidirect product of $U_1$ with the first
Frobenius kernel $\mathbb T_1$ of the torus. A torus weight restricts
trivially to $\mathbb T_1$ precisely when it is divisible by $p$.
Taking these weights gives $\ch M^{B_1}$ and proves the formula.
\end{proof}

Now fix $k\le j$ and apply the lemma in each finite internal degree of
$L_{jk}$. We retain the internal grading throughout the character
calculation and specialize it only after forming finite Hilbert
numerators. The expressions below display the resulting weight
characters, with $z$ still recording the torus weight.

First, we introduce the two factors
\[
 \mathcal K=\sum_{m=k+1}^{n}(-1)^{m-k-1}\binom nm z^{j-m} \qquad\text{and} \qquad
 \Sigma=(1-z^{-p})^n=\frac{\ch(S_+^p)}{\ch S^p}.
\]
Then the twisted Koszul resolution defining $L_{jk}$ in \S\ref{gr:kk} gives $\ch L_{jk}=\mathcal K\ch S$, and since $\ch S/\ch S^p=\ch Q$, the reduced character of $L_{jk}$
is therefore $\mathcal K\ch Q$.

For $\ell=0,1$, let
\[
 \mathcal T_\ell(z)=\sum_r\mathcal T_{j,k;\ell,r}z^{pr}
\]
be the finite numerator of the projected Tate character, so that
\[
 \frac{[\ch\widehat H^\ell(U_1,L_{jk})]_{p\mathbb Z}}{\ch S^p}
       =\mathcal T_\ell(z)\Sigma.
\]
The coefficients $\mathcal T_{j,k;\ell,r}$ are the Tate counts
calculated by the alternating sums in \eqref{gr:T}. The polynomial
factor before division is $\ch(S_+^p)$.
Define $Z^{\rm pre}$ to be the reduced character of the Borel
invariants, before induction to $G^{(1)}$. The lemma now gives
\[
 Z^{\rm pre}:=\frac{\ch L_{jk}^{B_1}}{\ch S^p}
 =\left[\mathcal T_0\Sigma+
 \frac{(1-z^2)\mathcal K\ch Q+(\mathcal T_1-\mathcal T_0)\Sigma}
      {1-z^{2p}}\right]_{p\mathbb Z}.
\]
All weights of $\ch S^p$ and of the displayed denominator are
divisible by $p$, so they commute with the weight projection.
Writing this over a common denominator gives
\[
 Z^{\rm pre}=\frac{\mathcal N}{1-z^{2p}} \qquad \text{and} \qquad
 \mathcal N=\left[(1-z^2)\mathcal K\ch Q+
          (\mathcal T_1-z^{2p}\mathcal T_0)\Sigma\right]_{p\mathbb Z}.
\]

As a Hilbert numerator of the finitely generated multigraded
$S^p$-module $L_{jk}^{B_1}$, $Z^{\rm pre}$ is a finite Laurent
polynomial, and thus the displayed division is exact.
The Tate counts have rows $0\le r\le n-3$ for $\mathcal T_0$ and
$1\le r\le n-2$ for $\mathcal T_1$. Since $\Sigma$ has normalized
weights in $[-n,0]$, the Tate part of $\mathcal N$ has normalized
support in $[-(n-1),n-1]$. All weights of
$(1-z^2)\mathcal K\ch Q$ lie strictly between $-np$ and $np$,
giving the same bounds after projection. Exact division therefore gives
\begin{equation}\label{det:support}
 \operatorname{supp}_z Z^{\rm pre}/p\subset[-(n-1),n-3].
\end{equation}

Write $Z_{j,k;c}=[z^{pc}]Z^{\rm pre}$ for its normalized weight
coefficients. To calculate them, let $P_{j,k;c}$ be the coefficient
of $z^{pc}$ in
\[
 \left[(1-z^2)\mathcal K\ch Q+
             (\mathcal T_1-\mathcal T_0)\Sigma\right]_{p\mathbb Z}.
\]
Expanding $\mathcal K$ and $\Sigma$ gives the finite formula
\eqref{gr:P}. This numerator has no normalized weight below $-n$.
In the required range $-(n-1)\le c\le n-3$, expansion of
$(1-z^{2p})^{-1}$ consequently uses only the terms
$\sum_{t=0}^{n-2}z^{2pt}$. Adding the coefficients of
$\mathcal T_0\Sigma$ gives exactly \eqref{gr:Z}.
Thus the character reconstruction requires only the finite sums
recorded in Appendix \ref{gr:appendix}.

For $k\le j$, \cite[Proposition~8.2, diagram~(8.5) and vanishing~(8.6)]{RSVdB2}
gives
\[
 K_{jk}^{G_1}\cong
 \operatorname{Ind}_{B^{(1)}}^{G^{(1)}}L_{jk}^{B_1}
 \qquad \text{and} \qquad R^1\operatorname{Ind}_{B^{(1)}}^{G^{(1)}}L_{jk}^{B_1}=0.
\]
Hence Euler induction computes the actual character. It sends a
normalized weight $c\ge0$ to $\chi_c$, sends weight $-1$ to zero,
and sends $c\le-2$ to $-\chi_{-c-2}$. The coefficient of $\chi_l$
after induction is therefore $Z_{j,k;l}-Z_{j,k;-l-2}$.

The decomposition in \cite[Theorem~4.3 and Proposition~9.1]{RSVdB2}
splits off the syzygy summands already counted by
$\mathcal O_{j,k;h,r}$, leaving $T\otimes S^p$ for a tilting module
$T$. Subtracting their reduced resolution characters from the induced
reduced character gives exactly the exterior-power subtraction in
\eqref{gr:KS}. By \eqref{det:support}, Euler induction involves only
highest weights at most $n-3$, and the reduced syzygy resolutions have the
same bound. Thus the residual tilting module has only the simple
summands $T(0),\ldots,T(n-3)$, since $n-3\le p-1$. Its character
therefore gives the covariant multiplicities in \eqref{gr:KS}.
For $k>j$ we use duality, as in \eqref{gr:duality}.

Ordinary cohomology contributes rows $1\le r\le n-3$. The potential
row zero forces tuple degree $-j<0$, and the potential upper row $n-2$
is excluded by the parity indicator forcing $e=1$ when $\ell=k$.

\subsection{Running example: the quadric cone}\label{gr:example}
For $n=4$ the family $\mathscr G_4$ consists of $C_4$, $S\{1\}$ and
$K\{1,1\}$, of ranks $1,2,2$. Here $C_4$ is the homogeneous coordinate ring of the smooth quadric
$\Gr(2,4)\subset\mathbb P^5$, and
$S\{1\}$ and $K\{1,1\}$ are its two spinor modules. The fusion atoms are
\[
 H_{4,0}=H_{4,p-2}=\frac{p(p^2-1)(p^2+1)}{30},\qquad
 H_{4,1}=H_{4,p-3}=\frac{p(p^2-1)(p^2-4)}{15},
\]
and $2^4-2=14$. For $n=4$, the sum in \eqref{gr:SS} has just
one term, $a^{(j)}_{3p-2}-a^{(j)}_{3p}$. For the source $C_4$
($j=0$), the weight counts give
\[
 a^{(0)}_0=\frac{p(4p^4-1)}3 \qquad \text{and} \qquad
 a^{(0)}_{3p-2}-a^{(0)}_{3p}=\frac{p(p^4-1)}{30},
\]
so
\eqref{gr:SC}--\eqref{gr:SK} give
\[
 a_1=\frac{p(4p^4-1)}3-14H_{4,0}=\frac{p(13p^4+2)}{15} \qquad \text{and} \qquad
 A_{S\{1\},C_4}=A_{K\{1,1\},C_4}=\frac{p(p^4-1)}{30}.
\]
For the source $S\{1\}$, the same calculation gives
\[
 a^{(1)}_0=\frac{8p(p^2-1)^2}{3} \qquad \text{and} \qquad
 a^{(1)}_{3p-2}-a^{(1)}_{3p}=\frac{p(p^4+10p^2+4)}{15}.
\]
We now compute the column of $K\{1,1\}$ directly by step~3 of
\S\ref{gr:mechanisms}.

Write $B=B_{1,1}$ and $E_{q,e}=E_{1;q,e}$. Then Equation~\eqref{gr:B}
gives
\[
 B=4H_{4,0}-H_{4,1}=\frac{p(p^2-1)(p^2+6)}{15}.
\]
The only syzygy target is $K\{1,1\}$, and the alternating terms in
\eqref{gr:O} cancel to give its multiplicity
\[
 o:=\mathcal O_{1,1;1,1}=B+E_{1,0}+E_{3,0}
   =\frac{p(p^4+10p^2+4)}{15},
\]
since
\[
 E_{1,0}=\binom{p+2}{3} \qquad \text{and} \qquad E_{3,0}=\binom p3.
\]
For the second identity, complement the four bounded exponents, so
total degree $3p-1$ becomes $p-3$.

Next, write $w$ for the normalized weight variable. Equation~\eqref{gr:T}
simplifies to
\[
 \sum_r\mathcal T_{1,1;0,r}w^r=B+E_{3,0}w
 \qquad \text{and} \qquad
 \sum_r\mathcal T_{1,1;1,r}w^r=(B+E_{3,0})w.
\]
Substituting these coefficients and the weight counts $W_4$ into
\eqref{gr:P}--\eqref{gr:Z} gives
\[
 Z_{1,1;0}-Z_{1,1;-2}=2p^3(p^2+1) \qquad \text{and} \qquad
 Z_{1,1;1}-Z_{1,1;-3}=-p^3.
\]
These are character coefficients before the syzygy subtraction, so
negative values are possible. The reduced resolution character of
$K\{1,1\}$ is $4\chi_0-\chi_1$. Thus \eqref{gr:KS} gives
\begin{align*}
 A_{S\{0\},K\{1,1\}}&=2p^3(p^2+1)-4o
    =\frac{2p(p^2-1)(13p^2+8)}{15}, \text{ and}\\
 A_{S\{1\},K\{1,1\}}&=-p^3+o
    =\frac{p(p^2-1)(p^2-4)}{15}.
\end{align*}
Together with $A_{K\{1,1\},K\{1,1\}}=o$, this completes the column.
As a check, it is the column of $S\{1\}$ with the two spinor entries
exchanged, as required by the outer automorphism $F\leftrightarrow F^*$
of $D_3=A_3$, which fixes $C_4$ and exchanges the spinor modules.
Altogether
\[
 A_4(p)=\frac p{15}
 \begin{pmatrix}
  13p^4+2 & 2(p^2-1)(13p^2+8) & 2(p^2-1)(13p^2+8)\\
  \tfrac12(p^4-1) & p^4+10p^2+4 & (p^2-1)(p^2-4)\\
  \tfrac12(p^4-1) & (p^2-1)(p^2-4) & p^4+10p^2+4
 \end{pmatrix},
\]
and its rank identity is $(1,2,2)A_4(p)=p^5(1,2,2)$.
Using the matrix $A_4(p)$, we calculate $a_e(C_4)$ for every $e$, and in turn calculate $s_p(C_4)$ for every odd prime $p$.
\begin{theorem}\label{gr:example-thm}
For every prime $p\ge3$,
\[
 s_p(C_4)=\frac{13p^2+8}{5(3p^2+2)},\qquad
 a_e(C_4)=\frac{13p^2+8}{5(3p^2+2)}p^{5e}
    +\frac{2(p^2+1)}{5(3p^2+2)}\left(\frac{p(p^2+2)}3\right)^e,
\]
The signature
increases to $13/15$ as $p \to \infty$.
\end{theorem}
\begin{proof}
The matrix commutes with the exchange of the two spinor coordinates. On
antisymmetric vectors it acts by $p^3$, while on symmetric vectors
$(v_0,v_1,v_1)$ it has eigenvalues $p^5$ and $p(p^2+2)/3$. The first row
of $A_4v=p^5v$ gives $v_1/v_0=(p^5-a_1)/(2A_{C_4,S\{1\}})
=(p^2+1)/(2(13p^2+8))$, and the rank normalization $v_0+4v_1=1$ gives
$s_p(C_4)=v_0$ by Theorem~\ref{thm:ffrt-spectrum}. Since $a_e(C_4)=(A_4^e)_{00}$
and the first coordinate vector lies in the invariant subspace of
symmetric vectors, $a_e(C_4)$ is a combination of $p^{5e}$ and $(p(p^2+2)/3)^e$, determined by
$a_0=1$ and $a_1$.
\end{proof}

\section{Polynomial formulas for the counting functions}\label{gr:poly}
Section~\ref{gr:section} and Appendix~\ref{gr:appendix} construct the
matrix from the weight counts $W_n$, the bounded-composition counts
$E_{j;q,e}$ and the fusion counts $H_{n,q}$. Here we evaluate these
functions by nonrecursive polynomial formulas. The cutoff estimates
ensure that the polynomial expression works throughout the stated prime
range, rather than only for sufficiently large primes. We then verify
the endpoint characteristic and deduce the degree bound, culminating in the proof of Theorem~\ref{gr:matrix}. For fixed $n$, these formulas determine $A_n(p)$ over $\mathbb Q[p]$ using only finite sums, products, binomial coefficients, and Bernoulli numbers. 

\subsection{Bounded-composition and weight counts}

Recall
\[
 C_{r,p}(M)=[x^M](1+x+\cdots+x^{p-1})^r,
\]
with a negative coefficient index interpreted as zero. The ordinary
$\operatorname{SL}_2$-weight multiplicity in the truncated polynomial
ring $Q=k[x_i,y_i]_{i=1}^n/(x_i^p,y_i^p)$ is
\begin{equation}\label{gr:weight}
 W_n(w)=C_{2n,p}\bigl(n(p-1)+w\bigr).
\end{equation}
Indeed, reflection of all $y$-exponents converts the weight
$\sum_i(\deg x_i-\deg y_i)=w$ into the displayed total degree.

\begin{lemma}\label{gr:binomial}
Fix $n\ge4$ and integers $\alpha,\beta$. Suppose that either
\[
 r=2n,\quad \text{for } -2n+1\le\beta\le-1
 \qquad\text{or}\qquad
 r=n,\quad \text{for } -n+1\le\beta\le-1.
\]
Then, for every integer $p\ge\max\{n-2,3\}$, we have
\begin{equation}\label{gr:C}
 C_{r,p}(\alpha p+\beta)
 =\sum_{\substack{0\le t\le r\\t<\alpha}}
 (-1)^t\binom rt
 \binom{(\alpha-t)p+\beta+r-1}{r-1},
\end{equation}
where the last binomial coefficient is evaluated as a polynomial in
its upper argument. In particular, this is a single polynomial in $p$,
with a summation range independent of $p$. The two cases cover the
weight counts and cohomology counts, respectively, in the construction of $A_n(p)$.
\end{lemma}
\begin{proof}
The identity
\[
 (1+x+\cdots+x^{p-1})^r=(1-x^p)^r(1-x)^{-r}
\]
gives the usual inclusion--exclusion formula
\[
 C_{r,p}(\alpha p+\beta)
 =\sum_{\substack{0\le t\le r\\(\alpha-t)p+\beta\ge0}}
 (-1)^t\binom rt
 \binom{(\alpha-t)p+\beta+r-1}{r-1}, 
\]
but its cutoff depends on $p$. We show that replacing this condition by
$t<\alpha$ leaves the value unchanged when the binomial coefficients
are interpreted polynomially.

Since $\beta<0$, every term with $t\ge\alpha$ has
$(\alpha-t)p+\beta<0$ and is excluded from both sums.
For $t<\alpha$, put $N=(\alpha-t)p+\beta$. If $N\ge0$, the term
occurs in both sums with the same value. If $N<0$, it occurs only in
\eqref{gr:C}, and we must show that its polynomial value is zero.

Because $\alpha-t\ge1$ and $p\ge n-2$, the hypotheses give
\[
 N\ge p+\beta\ge
 \begin{cases}
 -n-1\ge-(2n-1)=-(r-1),&r=2n,\\
 -1\ge-(n-1)=-(r-1),&r=n.
 \end{cases}
\]
Thus any such negative $N$ lies in $\{-(r-1),\ldots,-1\}$. For
these values the polynomial
\[
 \binom{N+r-1}{r-1}
 =\frac{(N+1)(N+2)\cdots(N+r-1)}{(r-1)!}
\]
vanishes. Every extra term introduced by the fixed cutoff therefore
contributes zero, proving \eqref{gr:C}. Its summands are polynomials
in $p$ and its range is fixed, proving polynomiality.
\end{proof}

\subsection{Fusion counts}
Recall that, with $V_0,\ldots,V_{p-2}$ the usual basis of the level-$(p-2)$
$\operatorname{SL}_2$ fusion ring,
\[
 H_{n,q}(p)=[V_q]\left(\sum_{i=0}^{p-2}V_i\right)^n.
\]
The endpoint label $q=p-1$ and its reflection $-1$ mean zero.
Write $\chi_q(z)=z^q+z^{q-2}+\cdots+z^{-q}$, also regarded as a
polynomial in $z+z^{-1}$. For $0\le q\le n-3$, define integers
$b_{n,q,m}$ by
\begin{equation}\label{gr:b}
 (x^{1-n}-x^{2-n})\chi_q(2-4x)
       =\sum_{m=1}^{n-1}b_{n,q,m}x^{-m}.
\end{equation}
Here $\chi_q(2-4x)$ means the polynomial in $z+z^{-1}$, evaluated
at $2-4x$. It can be expanded directly using
\[
 \chi_q(X)=\sum_{h=0}^{\lfloor q/2\rfloor}
          (-1)^h\binom{q-h}{h}X^{q-2h}.
\]
Recall $B_{2j}$ denotes the Bernoulli numbers, with $B_2=1/6$, and define
rational coefficients $c_{m,j}$ by the finite product
\begin{equation}\label{gr:cosecant-coefficients}
 \frac{1}{(2m-1)!}\prod_{h=1}^{m-1}(u+4h^2)
    =\sum_{j=0}^{m-1}c_{m,j}u^j\qquad(m\ge1).
\end{equation}
The empty product is $1$, so $c_{1,0}=1$.

\begin{lemma}\label{gr:fusion}
Put
$P_m(z)=\sum_{r=1}^{z-1}\csc^{2m}(r\pi/z)$ for positive integral
$z$. Its polynomial extension is
\begin{equation}\label{gr:cosecant-polynomial}
 P_m(z)=\sum_{j=0}^{m-1}c_{m,j}
 \frac{(-1)^j2^{2j+2}B_{2j+2}}{2j+2}(z^{2j+2}-1).
\end{equation}
The fusion counts are given by the finite polynomial formula
\begin{equation}\label{gr:H}
\begin{split}
 H_{n,q}(p)
 &=\frac{2^{2-n}}p\sum_{m=1}^{n-1}
       b_{n,q,m}\bigl(P_m(2p)-P_m(p)\bigr)\\
 &=2^{2-n}\sum_{m=1}^{n-1}\sum_{j=0}^{m-1}b_{n,q,m}c_{m,j}
 \frac{(-1)^j2^{2j+2}(2^{2j+2}-1)B_{2j+2}}{2j+2}\,p^{2j+1}.
\end{split}
\end{equation}
In particular $H_{n,q}\in\mathbb Q[p]$ has degree at most $2n-3$.
Moreover $H_{n,p-2-q}=H_{n,q}$, including the zero endpoint convention.
\end{lemma}
\begin{proof}
At $X_r=2\cos(r\pi/p)$, the sum of the fusion characters is
\[
 \sum_{i=0}^{p-2}\chi_i(X_r)
   =\frac{1-(-1)^r}{2-2\cos(r\pi/p)}.
\]
The sine-orthogonality projector for $V_q$, with
$x_r=\sin^2(r\pi/(2p))$, therefore gives
\[
 H_{n,q}(p)=\frac{2^{3-n}}p
 \sum_{\substack{1\le r<p\\r\ \mathrm{odd}}}
 (x_r^{1-n}-x_r^{2-n})\chi_q(2-4x_r).
\]
The Laurent polynomial in \eqref{gr:b} vanishes at $x=1$.
Reflection $r\mapsto2p-r$ doubles this sum, and the possible middle
term $r=p$ is zero. Separating odd and even indices proves
\eqref{gr:H}. The reflected fusion label has the same sine projector
on odd indices, proving the last assertion.

To prove \eqref{gr:cosecant-polynomial}, put $\partial=d/dx$ and
\[
 L_m(\partial)=\prod_{h=1}^{m-1}
 \frac{\partial^2+4h^2}{2h(2h+1)}
 =\sum_{j=0}^{m-1}c_{m,j}\partial^{2j}.
\]
The identity
$(\csc^{2m}x)''+4m^2\csc^{2m}x=2m(2m+1)\csc^{2m+2}x$
shows that $\csc^{2m}x=L_m(\partial)\csc^2x$.
Applying this operator to
$\sum_{r=0}^{z-1}\csc^2(x+r\pi/z)=z^2\csc^2(zx)$ and removing
the term $r=0$ expresses $P_m(z)$ as the constant Laurent coefficient of
$L_m(\partial)(z^2\csc^2(zx)-\csc^2x)$.
The Bernoulli expansion of $\cot x$, differentiated once, gives
\[
 (2j)!\,[x^{2j}]\csc^2x
   =\frac{(-1)^j2^{2j+2}B_{2j+2}}{2j+2}.
\]
Taking the constant coefficient proves
\eqref{gr:cosecant-polynomial}. Substitution gives the second line
of \eqref{gr:H}, whose largest possible exponent is $2n-3$.
\end{proof}

\subsection{The characteristic range and degree bound}
The character calculation in \S\ref{gr:cov} first assumed $p\ge n-1$.
The following lemma verifies the endpoint $p=n-2$ of the permissible range before we assemble
the polynomial formulas.

\begin{lemma}\label{gr:boundary}
The finite formulas \eqref{gr:SC}--\eqref{gr:SK} give the needed
multiplicities at $p=n-2$, whenever this is an odd prime, including
$n=5,p=3$.
\end{lemma}
\begin{proof}
Here $p\ge3$ and every source $T(j)$, for $0\le j\le n-3=p-1$,
is simple. The bound $n(p-1)+j\le np-3$ still restricts the
relevant linked weights to $\Lambda_n$. Write $\tau_\lambda$ for
the formal two-term characters used in the tilting inversion.
The tilting tensor-product formula gives
$\tau_\lambda=\ch T(\lambda)$ on this list with exactly one
exception. Put $H=p(p+1)$ and $L=p(p-1)$. Since
$T(H)=T(p)\otimes T(p)^{[1]}$, multiplication of
$\ch T(p)=\chi_p+\chi_{p-2}$ with its twist gives
\[
 \ch T(H)=\chi_H+\chi_{H-2}+\chi_L+\chi_{L-2}
          =\tau_H+\tau_L.
\]
Every other relevant weight has remaining tilting digit at most $p-1$.
Consequently the actual multiplicities agree with the formal inverse
$U_\lambda$ except that the coefficient at $L$ is $U_L-U_H$.
The covariant entries read only
\[
 U_0,\qquad U_{2p-2},\qquad \text{and} \qquad
 U_{(l+2)p-2}\quad(1\le l\le p-1),
\]
so the altered coefficient is unused. The formulas
\eqref{gr:SC}--\eqref{gr:SS} are exactly these same formal coefficients
written as finite sums, so they remain valid. Equivalently,
$T(L)^{G_1}=0$, since its first tilting digit is $p$ and
$T(p)^{G_1}=0$.

For $j\le p-2$, the covariant decomposition and fusion correction
therefore apply unchanged. At $j=p-1$, the Steinberg module is
$G_1$-projective, so higher cohomology vanishes and the decomposition
is tilt-free with coefficient module $(T(p-1)\otimes Q)^{G_1}$
\cite[Proposition~5.1 and its proof]{RSVdB2}.
The formulas give the same result: all fusion labels are $p-1$ or
$-1$, and $H_{n,p-1}=0$ by Lemma~\ref{gr:fusion}.

Finally, for the K-columns, the published cohomology and induction arguments
hold for $p\ge n-2$, while the residual tiltings have highest weight at most
$n-3=p-1$ and remain simple. The finite-truncation argument of
\S\ref{gr:ks} and the cutoff estimates above also hold at this endpoint. 
\end{proof}

Finally, we are ready to prove the main result of this section:

\begin{theorem}[Polynomiality of the multiplicities]\label{gr:matrix}
For fixed $n\ge4$, every entry of $A_n(p)$ agrees with a polynomial
in $\mathbb Q[p]$ of degree at most $d_n=2n-3$ throughout the prime
range $p\ge\max\{n-2,3\}$.
The entries are given by \eqref{gr:SC}--\eqref{gr:SK} for covariant
sources and by \eqref{gr:KK}--\eqref{gr:KS} and \eqref{gr:duality}
for syzygy sources.
\end{theorem}

\begin{proof}
Section~\ref{gr:section} and Appendix~\ref{gr:appendix} express every
entry as a finite linear combination, with coefficients and ranges
independent of $p$, of $W_n$, $E_{j;q,e}$ and $H_{n,q}$.
Every weight query has constant weight in $[-(n-1),n-1]$, so its
bounded-composition index has constant part in $[-2n+1,-1]$, while
every $E$-index has constant part $-j-2e\in[-n+1,-1]$.
Lemma~\ref{gr:binomial} therefore supplies a polynomial for each such
count throughout the prime range. All fusion labels reduce to
$0\le q\le n-3$, and \eqref{gr:H} supplies their polynomials.
The support argument in \S\ref{gr:ks} makes every coefficient
extraction a finite sum with range depending only on $n$.

Substituting these formulas proves polynomiality of every entry.
Section~\ref{gr:section} and Lemma~\ref{gr:boundary} identify these
polynomials with the actual multiplicities at every admissible prime.
Finally, the rank identity $\rho A_n(p)=p^{d_n}\rho$ and positivity
of the ranks give the degree bound by
Corollary~\ref{cor:ffrt-rational}.
\end{proof}

\section{Characteristic limits}\label{sig:section}
For each fixed prime, Section~\ref{sec:ffrt} obtains the signature from
the normalized Perron eigenvector. In this section we show $\displaystyle \lim_{p \to \infty} s_p(C_n)=\lim_{p \to \infty} \dfrac{a_1(C_n)}{p^{d_n}}$ and use this to give an explicit calculation for the limit $F$-signature. 
Since the Pl\"ucker
rings are Gorenstein, this comparison follows from the socle formula
for splitting numbers and Smirnov's uniform Hilbert--Kunz estimate \cite{SmirnovFamilies}.
The polynomiality established in Section~\ref{gr:poly} then identifies
the limit with a leading coefficient, which we evaluate explicitly.

\subsection{Comparison with the first free multiplicity}\label{lim:section}
Fix $n\ge4$ and recall $d_n:=\dim(C_n)=2n-3$.
For an ideal $I$ primary to the homogeneous maximal ideal of $C_n$, write
\[
 e_{\mathrm{HK}}(I)
 =\lim_{e\to\infty}
   \frac{\ell(C_n/I^{[p^e]})}{p^{ed_n}}
\]
for its Hilbert--Kunz multiplicity. 

\begin{theorem}\label{lim:comparison}
For each fixed $n\ge4$, we have
\begin{equation}\label{lim:first-free-comparison}
 s_p(C_n)=\frac{a_1(C_n)}{p^{2n-3}}+O(p^{-1}) \qquad \text{and} \qquad
 \lim_{p\to\infty}s_p(C_n)=[p^{d_n}]a_1(C_n).
\end{equation}
The implied constant may depend on $n$, but is independent of $p$.
\end{theorem}

\begin{proof}
Let $\mathscr C_n$ be the integral Pl\"ucker algebra. Choose a
homogeneous system of parameters $x_1,\ldots,x_{d_n}$ in
$\mathscr C_n\otimes\mathbb Q$, and choose a homogeneous element $u$
whose class generates the socle modulo $Q=(x_1,\ldots,x_{d_n})$.
After inverting a nonzero integer $L$, these elements are defined over
$B=\mathbb Z[1/L]$. Enlarging $L$ makes $\mathscr C_n/Q$ finite free
over $B$ and ensures that the class of $u$ remains a nonzero socle
element in every fiber. Indeed, the finitely many equations expressing
annihilation by the degree-one generators hold after clearing
denominators, and the nonzero class of $u$ can be made part of a
$B$-basis. Each fiber is Gorenstein, so this class generates its socle.
Set $J=(Q,u)$.

For $p\nmid L$, put $R_p=\mathscr C_n\otimes_B\mathbb F_p$, and
write $Q_p,J_p,u_p$ for the reductions of these choices. The
Gorenstein socle formula \cite[Corollary~6.8(i)]{PolstraTucker},
which gives the Hilbert--Kunz description of $F$-signature from
\cite{HL}, yields
\begin{align}
 a_1(R_p)
 &=\ell\bigl(R_p/(Q_p^{[p]}:u_p^p)\bigr)
   =\ell(R_p/Q_p^{[p]})-\ell(R_p/J_p^{[p]}),
   \label{lim:socle-splitting}\\
 s(R_p)&=e_{\mathrm{HK}}(Q_p)-e_{\mathrm{HK}}(J_p).
   \label{lim:socle-signature}
\end{align}
The second equality in \eqref{lim:socle-splitting} follows from
multiplication by $u_p^p$ modulo $Q_p^{[p]}$. These formulas also show
that the quantities are unchanged upon extending the perfect ground
field, so they compute the multiplicity and signature of $C_n$.

The algebra $\mathscr C_n$ is flat over $B$, with geometrically
integral fibers of dimension $d_n$, and $\mathscr C_n/J$ is finite over
$B$. Thus $B\to\mathscr C_n$ with the ideal $J$ is an affine
$J$-family in the sense of \cite[Definition~3.8]{SmirnovFamilies}.
Smirnov's uniform estimate \cite[Corollary~4.9]{SmirnovFamilies}
therefore gives, after excluding finitely many further primes,
\[
 \left|e_{\mathrm{HK}}(J_p)
       -\frac{\ell(R_p/J_p^{[p]})}{p^{d_n}}\right|
 \le\frac{C}{p},
\]
where $C$ is independent of $p$. Since $Q_p$ is generated by a regular
sequence,
\[
 \frac{\ell(R_p/Q_p^{[p]})}{p^{d_n}}
 =\ell(R_p/Q_p)=e_{\mathrm{HK}}(Q_p).
\]
Subtracting \eqref{lim:socle-splitting}, divided by $p^{d_n}$, from
\eqref{lim:socle-signature} proves the comparison. Finally,
$a_1(C_n)$ is a polynomial in $p$ of degree at most $d_n$ by
Section~\ref{gr:poly}, so its normalized limit is its coefficient of
$p^{d_n}$.
\end{proof}

The comparison itself uses Gorensteinness and uniform Hilbert--Kunz
convergence, and does not require FFRT. The Frobenius decomposition supplies
the explicit polynomial $a_1(C_n)$ that we will use to evaluate the limit.

\subsection{The first free multiplicity and the limit}

\begin{proposition}\label{gr:first-free}
The first free multiplicity is the explicit polynomial
\begin{equation}\label{gr:a1}
 a_1(C_n)=C_{2n,p}\bigl(n(p-1)\bigr)
       -C_{2n,p}\bigl(n(p-1)+2\bigr)
       -(2^n-2)H_{n,0}(p).
\end{equation}
\end{proposition}
\begin{proof}
Set $j=0$ in \eqref{gr:SC} and use
\eqref{gr:weyl-coefficients} and \eqref{gr:weight}. The endpoint
characteristic is included by Lemma~\ref{gr:boundary}.
\end{proof}

Let $B_{2r}$ denote the Bernoulli numbers, with $B_2=1/6$.
Define
\[
 \gamma_n=
 \frac{2^n(2^{2n-2}-1)(-1)^nB_{2n-2}}{(2n-2)!}.
\]
\begin{corollary}\label{gr:limit}
For every fixed $n\ge4$,
\begin{equation}\label{gr:limit-formula}
 \lim_{p\to\infty}s(C_n)=
 -\frac{2}{(2n-3)!}
 \sum_{k=0}^{n-1}(-1)^k\binom{2n}{k}(n-k)^{2n-3}
 -(2^n-2)\gamma_n.
\end{equation}
In particular, for $n=4,5,6,7,8$ the
values are respectively
\[
 \frac{13}{15},\quad\frac{97}{126},\quad
 \frac{15067}{22680},\quad\frac{2983}{5280}, \quad  
 \frac{11677453}{24324300}.
\]
\end{corollary}
\begin{proof}
By Theorem~\ref{lim:comparison}, the required limit is
$[p^{2n-3}]a_1(C_n)$. We compute this coefficient separately for the
two contributions in \eqref{gr:a1}.

First consider the bounded-composition terms. Put $a=n(p-1)$.
The symmetry $C_{2n,p}(a-2)=C_{2n,p}(a+2)$ gives
\[
 C_{2n,p}(a)-C_{2n,p}(a+2)
 =-\frac12\bigl(C_{2n,p}(a+2)-2C_{2n,p}(a)+C_{2n,p}(a-2)\bigr).
\]
For sufficiently large $p$, formula~\eqref{gr:C} gives, for
$h=-2,0,2$,
\[
 C_{2n,p}(a+h)
 =\sum_{k=0}^{n-1}(-1)^k\binom{2n}{k}
   \binom{(n-k)p+n-1+h}{2n-1}.
\]
To calculate the indicated combination, regard
$f(X)=\binom{X}{2n-1}$ as a polynomial in $X$. Expanding its three
translates shows that
\[
 f(X+2)-2f(X)+f(X-2)
 =\frac{4}{(2n-3)!}X^{2n-3}
   +\text{terms of degree at most }2n-4.
\]
Indeed, the terms in degrees $2n-1$ and $2n-2$ cancel, and the
displayed coefficient comes from the leading term
$X^{2n-1}/(2n-1)!$ of $f$. Substituting
$X=(n-k)p+n-1$, multiplying by $-1/2$, and summing gives
\[
 [p^{2n-3}]\bigl(C_{2n,p}(a)-C_{2n,p}(a+2)\bigr)
 =-\frac{2}{(2n-3)!}
   \sum_{k=0}^{n-1}(-1)^k\binom{2n}{k}(n-k)^{2n-3}.
\]

Next consider $H_{n,0}(p)$. Since $\chi_0=1$, equation~\eqref{gr:b}
has only the two nonzero coefficients
$b_{n,0,n-1}=1$ and $b_{n,0,n-2}=-1$. Thus \eqref{gr:H} becomes
\[
 H_{n,0}(p)=\frac{2^{2-n}}p
 \bigl(P_{n-1}(2p)-P_{n-1}(p)-P_{n-2}(2p)+P_{n-2}(p)\bigr).
\]
The terms involving $P_{n-2}$ have degree at most $2n-5$ after
division by $p$, so they do not contribute to $[p^{2n-3}]H_{n,0}$.
By \eqref{gr:cosecant-polynomial} and
$c_{n-1,n-2}=1/(2n-3)!$, the leading coefficient of $P_{n-1}(z)$ is
\[
 [z^{2n-2}]P_{n-1}(z)
 =\frac{(-1)^n2^{2n-2}B_{2n-2}}{(2n-2)!}.
\]
Consequently
\[
 [p^{2n-3}]H_{n,0}(p)
 =2^{2-n}(2^{2n-2}-1)
   \frac{(-1)^n2^{2n-2}B_{2n-2}}{(2n-2)!}
 =\gamma_n.
\]
Combining the two coefficients according to \eqref{gr:a1} proves
\eqref{gr:limit-formula}.
\end{proof}

\begin{corollary}\label{gr:large-n}
Write $s_\infty(C_n)=\lim_{p\to\infty}s(C_n)$. Then
\[
 s_\infty(C_n)=12\sqrt{\frac3\pi}\,n^{-3/2}
 \left(1-\frac{3}{8n}+O(n^{-2})\right)
 \qquad(n\to\infty).
\]
In particular, $\lim_{n\to\infty}s_\infty(C_n)=0$.
For illustration, evaluation of \eqref{gr:limit-formula} gives the
following values, rounded to eight decimal places:
\begin{center}
\begin{tabular}{@{}lrrrr@{}}
\toprule
$n$ & $10$ & $20$ & $50$ & $100$\\
$s_\infty(C_n)$ & $0.35473508$ & $0.12866237$ & $0.03291929$ & $0.01168254$\\
\bottomrule
\end{tabular}
\end{center}
\end{corollary}
\begin{proof}
Let $M_n$ be the first term on the right of \eqref{gr:limit-formula},
and let $f_n$ be the density of the sum of $2n$ independent uniform
random variables on $[-1/2,1/2]$. The convolution formula for this
density gives $M_n=-2f_n''(0)$. Fourier inversion, followed by the
change of variable $t\mapsto2t$, yields
\[
 M_n=\frac{16}{\pi}\int_0^\infty
 t^2\left(\frac{\sin t}{t}\right)^{2n}\,dt.
\]
For every fixed $\delta>0$, the integral over $[\delta,\infty)$ is
exponentially small. Indeed, $|\sin t/t|$ is bounded there by a constant
less than one, and $t^2|\sin t/t|^8$ is integrable.
Near zero, $\log(\sin t/t)=-t^2/6+O(t^4)$.
Substituting $t=u/\sqrt n$ and using dominated convergence therefore gives
\[
 \lim_{n\to\infty}n^{3/2}M_n
 =\frac{16}{\pi}\int_0^\infty u^2e^{-u^2/3}\,du
 =12\sqrt{\frac3\pi}.
\]
Keeping one more term, $\log(\sin t/t)=-t^2/6-t^4/180+O(t^6)$, the same
substitution gives $(\sin t/t)^{2n}=e^{-u^2/3}\bigl(1-u^4/(90n)\bigr)+
O(n^{-2})$ in the relevant range, with a dominated error. Since
\[
 \int_0^\infty u^6e^{-u^2/3}\,du=\frac{135}4\int_0^\infty u^2e^{-u^2/3}\,du
 \qquad\text{and}\qquad \frac{135}{4}\cdot\frac1{90}=\frac38,
\]
this yields the factor $1-\frac{3}{8n}+O(n^{-2})$.
Finally, the Bernoulli--zeta identity
$(-1)^nB_{2n-2}=2(2n-2)!\zeta(2n-2)/(2\pi)^{2n-2}$ shows that
\[
 (2^n-2)\gamma_n
 =O\!\left(\left(\frac4{\pi^2}\right)^n\right),
\]
and this correction does not affect the asymptotic.
\end{proof}

\begin{remark}\label{gr:shideler}
Over a field of characteristic different from two, $C_4$ is the cone over
a quadric in six variables, which is defined by a sum of six squares. For
a sum of $2m$ squares, Theorem~24 and Corollary~22 of \cite{Shideler}
give the limit $F$-signature
$1+(-1)^m2^{2m}(2^{2m}-1)B_{2m}/(2m)!$; at $m=3$ this is $13/15$, the
value of \eqref{gr:limit-formula} at $n=4$, and we note that the two expressions come
from unrelated arguments.
\end{remark}

\section{Riemann--Roch reduction and the Frobenius spectrum}\label{gr:rr-reduction}
We first record the matrix consequence of the standard Frobenius
compatibility in Riemann--Roch \cite[(2.8)]{KuranoRR}; see also
\cite[Section~4]{KuranoOhta} for its use with FFRT summand classes.
Here $G_0(Y)$ is the Grothendieck group of coherent sheaves, and
$\operatorname{CH}_i(Y)$ is the Chow group of dimension-$i$ cycles.

\begin{proposition}[Riemann--Roch eigenrows]\label{rr:general}
Let $Y$ be a $d$-dimensional quasi-projective variety over a perfect
field of characteristic $p>0$, and let $\mathcal E_0,\ldots,\mathcal E_{h-1}$
be coherent sheaves with
$F_*\mathcal E_j\cong\bigoplus_i\mathcal E_i^{\oplus A_{ij}}$.
Let $\tau_i:G_0(Y)_{\mathbb Q}\to\operatorname{CH}_i(Y)_{\mathbb Q}$
be the dimension-$i$ Riemann--Roch component, and put
\[
 V_i=\operatorname{span}_{\mathbb Q}\{\tau_i([\mathcal E_j]):0\le j<h\},
 \qquad r_i=\dim_{\mathbb Q}V_i.
\]
For a basis $\ell_{i,1},\ldots,\ell_{i,r_i}$ of $V_i^*$, the rows
\[
 w_{i,a}=\bigl(\ell_{i,a}(\tau_i([\mathcal E_j]))\bigr)_{j=0}^{h-1}
 \quad\text{satisfy}\quad w_{i,a}A=p^iw_{i,a}.
\]
All $\sum_i r_i$ rows are linearly independent. Thus every $i$ with
$r_i>0$ gives an eigenvalue $p^i$ of geometric multiplicity at least
$r_i$, and different dimensions $i$ give distinct eigenvalues.
\end{proposition}
\begin{proof}
Frobenius compatibility gives
$\tau_i(F_*\alpha)=p^i\tau_i(\alpha)$.
Applying $\ell_{i,a}\tau_i$ to the decompositions proves the eigenrow
identity. For fixed $i$, evaluation on the spanning set
$\{\tau_i([\mathcal E_j])\}_j$ embeds $V_i^*$ into $\mathbb Q^h$,
so these $r_i$ rows are independent. Since the scalars $p^i$ are
distinct, the rows from different dimensions are independent as well.
\end{proof}

For smooth $Y$, the class $\tau([\mathcal E])$ is
$\ch(\mathcal E)\operatorname{td}(T_Y)\cap[Y]$.
We now compute these classes on the punctured Pl\"ucker cone and prove
independence of the resulting eigenrows. This proves
Theorem~\ref{thm:B}(2), independently of Section~\ref{sig:section}. Put
\[
 d=2n-3,\qquad q=\lfloor n/2\rfloor,\qquad \text{and} \qquad
 m=n-2+\left\lceil\frac{(n-3)^2}{2}\right\rceil.
\]
Let $B=B_n(p)$ be the $m\times m$ matrix obtained from $A_n(p)$ by
recording total multiplicities on the duality orbits
\[
 S\{j\}^{\vee}=S\{j\},\qquad
 K\{j,k\}^{\vee}=K\{n-2-j,n-2-k\}.
\]
Thus each entry of $B$ is the sum over a target orbit for one source
representative, as in \eqref{eq:duality-quotient}. Here $\rho$ denotes
the rank row on these orbits. The theorem below supplies $q<m$
independent eigenrows, determining part of the spectrum of $B$.

\begin{theorem}\label{gr:rr-eigenrows}
Under the hypotheses of Theorem~\ref{gr:matrix}, $B$ has $q$
linearly independent rational left eigenrows $w_0,\ldots,w_{q-1}$,
independent of $p$, satisfying
\begin{equation}\label{gr:rr-eigenvalues}
 w_iB=p^{d-2i}w_i,\qquad w_0=\rho.
\end{equation}
They are given by
\begin{align}
 T_n(x)&=\left(\frac{x}{2\sinh(x/2)}\right)^{2n}
                 \left(\frac{\sinh x}{x}\right)^2,
 \label{gr:rr-todd}\\
 w_i(S\{j\})&=[x^{2i}]\,T_n(x)
                        \sum_{a=0}^j e^{(j-2a)x},
 \label{gr:rr-covariant}\\
 w_i(K\{j,k\})&=
 \sum_{a=k+1}^{n-1}(-1)^{a-k-1}\binom n{b_a}
                              w_i(S\{c_a\}),
 \label{gr:rr-syzygy}
\end{align}
where $(c_a,b_a)=(j-a,a)$ for $a\le j$, and
$(c_a,b_a)=(a-j-1,a+1)$ for $a>j$.
These values are constant on each duality orbit.
Consequently the normalized Perron eigenvector satisfies
\begin{equation}\label{gr:rr-constraints}
 w_0v=1,\qquad w_iv=0\quad(1\le i<q),
\end{equation}
and $s_p(C_n)$ is determined by a linear system with
$m-q=\binom{n-2}{2}$ unknowns.
\end{theorem}

\begin{proof}
Write $X=\Gr(2,n)$ and let $\pi:U\longrightarrow X$ be the punctured
Pl\"ucker cone, the complement of the zero section in $\mathcal O_X(-1)$.
With $h=c_1(\mathcal O_X(1))$, localization for this line bundle gives
\begin{equation}\label{gr:rr-chow}
 \operatorname{CH}^*(U)_{\mathbb Q}
 =\operatorname{CH}^*(X)_{\mathbb Q}/(h)
 =\mathbb Q[u]/(u^q),\qquad \deg u=2.
\end{equation}
For the last equality use the rank-two quotient bundle $\mathcal Q$:
the usual presentation has generators $c_1(\mathcal Q),c_2(\mathcal Q)$
and relations the complete symmetric polynomials of degrees $n-1$ and $n$.
Setting $c_1(\mathcal Q)=h=0$ makes its formal roots $x,-x$, while
the relations become $u^q=0$ with $u=x^2=-c_2(\mathcal Q)$.
This is the same hyperplane quotient that computes the local cone's
Chow groups in \cite[Theorem~2.3]{KuranoSingh}.

Every module in $\mathscr G_n$ restricts to a vector bundle on $U$,
since it occurs in a Frobenius pushforward and $U$ is smooth.
For such a module $M$, define $w_i(M)$ to be the coefficient of $u^i$
in
\[
 \ch(M|_U)\operatorname{td}(T_U)
       \in\operatorname{CH}^*(U)_{\mathbb Q}.
\]
The relative tangent bundle of $\pi$ is trivial, so
$\operatorname{td}(T_U)=\pi^*\operatorname{td}(T_X)$.
From the tautological sequence and
$T_X=\mathcal R^{\vee}\otimes\mathcal Q$ one obtains in $K_0(X)$
\[
 [T_X]=n[\mathcal Q]-[\mathcal Q^{\vee}\otimes\mathcal Q].
\]
Modulo $h$, the roots of the second term are $0,0,2x,-2x$.
The Todd formula therefore gives \eqref{gr:rr-todd}.
The bundle associated with $S\{j\}$ is $\Sym^j\mathcal Q$, up to
grading normalization \cite[Proposition~11.2]{RSVdB2}, and its character
gives \eqref{gr:rr-covariant}.

For completeness, the syzygy formula requires exactness only on $U$.
The full-rank locus of $n\times2$ matrices is a principal
$\SL_2$-bundle over $U$. Restrict the equivariant resolution
\cite[(3.5), (3.7)]{RSVdB2} to this locus and descend its exact tail
resolving $\Omega^{k+1}M_j$. Descent along a torsor is exact, and
its term in homological degree $a$ is, after forgetting a graded
shift, $\binom n{b_a}$ copies of the bundle for $S\{c_a\}$.
Graded shifts are trivial after restriction to $U$. Equivalently,
line bundle twists disappear in \eqref{gr:rr-chow}.
Taking the alternating class of this tail proves
\eqref{gr:rr-syzygy}. All the symmetric or divided powers involved
have degree at most $n-3<p$, so this uses precisely the stated
characteristic bound. No exactness of the global invariant functor
is being asserted.

Applying Proposition~\ref{rr:general} on $U$ to the dimension-$(d-2i)$
components gives \eqref{gr:rr-eigenvalues} before aggregation.
Ordinary module duality restricts to bundle duality on $U$.
Only even codimensions survive in \eqref{gr:rr-chow}, and bundle
duality fixes these Chern-character components. Hence the rows are
constant on duality orbits and descend to $B$.

Finally,
\[
 \sum_{a=0}^j e^{(j-2a)x}
       =\frac{\sinh((j+1)x)}{\sinh x}
\]
shows that $w_i(S\{j\})=(j+1)P_i((j+1)^2)$, where $P_i$ has degree
$i$ and leading coefficient $1/(2i+1)!$. The minor on the columns
$S\{0\},\ldots,S\{q-1\}$ is consequently
\[
 \det\bigl(w_i(S\{j\})\bigr)_{0\le i,j<q}
 =\frac{\displaystyle\prod_{a=1}^q a
               \prod_{1\le a<b\le q}(b^2-a^2)}
             {\displaystyle\prod_{i=0}^{q-1}(2i+1)!}\ne0,
\]
which proves independence. Applying the rows to $Bv=p^dv$ and using
rank normalization gives \eqref{gr:rr-constraints}. Eliminating these
$q$ covariant coordinates gives the asserted number of unknowns.
\end{proof}

Here is an explicit form of the resulting smaller system. Let $W$ have
rows $w_i$, let $E$ index $S\{0\},\ldots,S\{q-1\}$, and let $J$ index
the remaining duality orbits. Partition $W=(W_E\;W_J)$ and $B$
accordingly, and put $e=(1,0,\ldots,0)^t\in\mathbb Q^q$. Set
\begin{equation}\label{gr:rr-residual}
 H=B_{JJ}-B_{JE}W_E^{-1}W_J,\qquad
 b=B_{JE}W_E^{-1}e,\qquad M=p^dI-H.
\end{equation}
Then the coordinates $z=v_J$ satisfy
\begin{equation}\label{gr:rr-signature}
 Mz=b,\qquad
 s_p(C_n)=e^tW_E^{-1}(e-W_Jz).
\end{equation}
Indeed $v_E=W_E^{-1}(e-W_Jz)$, and the $J$-rows of the eigenvector
equation give \eqref{gr:rr-residual}. The other rows follow from
$WB=\operatorname{diag}(p^{d-2i})W$ and the invertibility of $W_E$.
In coordinates $(Wv,v_J)$ the matrix $B$ is block triangular, and so
\begin{equation}\label{gr:rr-factorization}
 \det(tI-B)=
 \prod_{i=0}^{q-1}(t-p^{d-2i})\det(tI-H).
\end{equation}
The simplicity of the Perron root implies $\det M\ne0$. Thus
\eqref{gr:rr-signature}, or Cramer's rule applied to it, is an exact
uniform formula of size $\binom{n-2}{2}$. This gives sizes $1,3,6$
for $n=4,5,6$. It does not assert that this size is minimal or supply
a scalar evaluation of the remaining determinant for arbitrary values of $n$.

\subsection{The full characteristic polynomial}\label{gr:characteristic}
The quotient calculation above accounts for $q$ eigenvalues. We now
construct the remaining characteristic factors directly from the
bounded-composition atoms $E_{j;s,e}$ of \eqref{gr:E}, without first
assembling $A_n(p)$. Put $\nu=n-3$ and let $\mathcal V_n$ be the rational
vector space freely generated by $e_a=e_{S\{a\}}$ and
$f_{jk}=e_{K\{j,k\}}$. Let $D$ be the linear involution induced by
module duality: it fixes the $e_a$ and sends $f_{jk}$ to
$f_{\bar j,\bar k}$, where $\bar j=\nu+1-j$. By
\eqref{gr:duality}, it commutes with $A_n(p)$.

Using the resolution \eqref{eq:defining-resolution} and its indices
$(c_a,b_a)$ in \eqref{gr:rr-syzygy}, set
\begin{equation}\label{gr:boundary-classes}
 t_{jk}=\sum_{a=k+1}^{n-1}(-1)^{a-k-1}\binom n{b_a}e_{c_a},
 \qquad x_{jk}=f_{jk}-t_{jk}\qquad(1\le j,k\le\nu).
\end{equation}
Exactness of the tail on the punctured cone, proved in
Theorem~\ref{gr:rr-eigenrows}, places every $x_{jk}$ in
$\mathcal N_n=\ker(\mathcal V_n\longrightarrow K_0(U)_{\mathbb Q})$.
These are formal signed classes, and their coordinates need not be
nonnegative.

\begin{lemma}\label{gr:boundary-bases}
The duality eigenspaces $\mathcal N_n^+$ and $\mathcal N_n^-$ have bases
\begin{align*}
 \mathcal B^+&=(y_{jk}:1\le k\le j\le\nu),
 &y_{jk}&=\tfrac12(x_{jk}+Dx_{jk}),\\
 \mathcal B^-&=(a_{jk}:1\le k<j\le\nu)\cup
 (a_{jj}:1\le j\le\lfloor\nu/2\rfloor),
 &a_{jk}&=x_{jk}-Dx_{jk}=f_{jk}-f_{\bar j,\bar k}.
\end{align*}
Their dimensions are $d_+=\nu(\nu+1)/2$ and
$d_-=\lfloor\nu^2/2\rfloor$, respectively.
\end{lemma}
\begin{proof}
Theorem~\ref{gr:rr-eigenrows} shows that the covariants span the
$q$-dimensional quotient and that duality acts trivially there.
Counting the duality orbits in $\mathcal V_n$ therefore gives the
stated dimensions. The negative vectors have disjoint antisymmetric
syzygy coordinates. For the positive vectors, the strictly
lower-triangular pairs likewise have distinct syzygy coordinates.
The only repeated diagonal coordinates are those of $y_{jj}$ and
$y_{\bar j,\bar j}$, whose difference is
$t_{\bar j,\bar j}-t_{jj}$. For $1\le j\le\lfloor\nu/2\rfloor$,
these differences have distinct highest covariant indices $n-j-2>j$
with nonzero coefficients, by \eqref{gr:boundary-classes}. Both
families are consequently independent, and their sizes prove the claim.
\end{proof}

For $1\le k\le j\le\nu$, let $\mathcal E_{jk}(p)$ be the column
obtained by retaining only the $E$-atoms in the linear reconstruction
of Appendix~\ref{gr:appendix}. Precisely, omit every $W_n$ and $H_{n,r}$
term in \eqref{gr:T}--\eqref{gr:KS}, but retain the cohomology assembly,
Jordan reconstruction, Euler induction, and covariant tail subtraction.
The resulting finite formula is recorded in \eqref{gr:boundary-vector}.

\begin{lemma}[Boundary transfer]\label{gr:boundary-transfer}
For $1\le k\le j\le\nu$, we have
\begin{equation}\label{gr:boundary-action}
 A_n(p)x_{jk}=\mathcal E_{jk}(p).
\end{equation}
This equality includes every covariant output coordinate.
\end{lemma}

\begin{proof}
Since $x_{jk}=f_{jk}-t_{jk}$, the required identity is equivalent to
\[
 A_n(p)f_{jk}=A_n(p)t_{jk}+\mathcal E_{jk}(p).
\]
Every operation in the column reconstruction
\eqref{gr:T}--\eqref{gr:KS} is linear in the $W_n$, $H_{n,s}$, and
$E_{j;q,e}$ inputs. Let $\mathcal G_{jk}(p)$ be the column obtained by
omitting the $E$-inputs and retaining the $W_n$ and $H_{n,s}$ inputs.
By definition of $\mathcal E_{jk}$,
\[
 A_n(p)f_{jk}=\mathcal G_{jk}(p)+\mathcal E_{jk}(p).
\]
We will show that $\mathcal G_{jk}(p)=A_n(p)t_{jk}$. Both columns
include the covariant outputs obtained after subtracting the syzygy
characters.

\smallskip\noindent
\emph{Matching the weight and fusion inputs.}
The Borel numerator of the syzygy source is
\[
 \mathcal K_{jk}(z)
 =\sum_{b=k+1}^{n}(-1)^{b-k-1}\binom nb z^{j-b}.
\]
Write $\mathcal W_p(P)$ for the bounded-weight contribution to the
induced character obtained from a numerator $P(z)$ by the projection,
Jordan reconstruction, and Euler induction of \S\ref{gr:ks}.
Put $w=z^p$. The reflection identities are
\[
 \mathcal W_p(z^{-a-2})=-\mathcal W_p(z^a),\qquad
 \mathcal W_p(z^{-1})=0\qquad(a\ge0).
\]
To see the first identity, use the symmetry of $\ch Q$ to note that
$(1-z^2)(z^a+z^{-a-2})\ch Q$ changes sign under $z\mapsto z^{-1}$.
Projection to weights divisible by $p$ preserves this property.
After division by $1-w^2$, the coefficients at $w^l$ and
$w^{-l-2}$ agree. These paired terms cancel under Euler induction,
which sends them to opposite characters; weight $-1$ contributes zero.
The same argument applied to $(1-z^2)z^{-1}\ch Q$ proves the second
identity.

In $\mathcal K_{jk}$, the term with $b=j+1$ therefore contributes
zero. Reflecting the terms with $b\ge j+2$ and retaining those with
$b\le j$ gives
\[
 \mathcal W_p(\mathcal K_{jk})
 =\sum_{a=k+1}^{n-1}(-1)^{a-k-1}\binom n{b_a}
    \mathcal W_p(z^{c_a}),
\]
with exactly the indices $(c_a,b_a)$ defining $t_{jk}$ in
\eqref{gr:boundary-classes}.
Likewise, reindexing \eqref{gr:B} gives
\[
 B_{j,k}
 =\sum_{a=k+1}^{n-1}(-1)^{a-k-1}\binom n{b_a}H_{n,c_a}.
\]
After omitting the $E$-terms, equations~\eqref{gr:T}--\eqref{gr:O}
give normalized Tate rows $B_{j,k}$ and $wB_{j,k}$, and syzygy
multiplicity $B_{j,k}$ on every diagonal target, with zero on the
other syzygy targets. Thus both kinds of input have the alternating
coefficients of the covariant tail.

\smallskip\noindent
\emph{Identifying these inputs with actual covariant columns.}
We must also justify that Euler induction computes the actual
covariant character, rather than a difference involving higher
induction. Fix $0\le s\le\nu\le p-1$ and put
$L_s=(s\omega)\otimes S$. Tensor identity and Kempf vanishing give
\[
 R\Ind_B^G L_s=T(s)\otimes S.
\]
Twisting \cite[Lemma~10.2 and Proposition~10.3]{RSVdB1} by $s\omega$
selects fusion labels $s$ and $p-2-s$ in even and odd Tate degree.
Their multiplicities agree by Lemma~\ref{gr:fusion}, giving normalized
Tate numerators $H_{n,s}$ and $wH_{n,s}$ in degrees zero and one.
The corresponding normalized weights are dominant, so their first
higher inductions vanish. For $s=p-1$, both Tate terms vanish.

In the remainder of this paragraph, write
$\mathsf I=\Ind_{B^{(1)}}^{G^{(1)}}$.
The ordinary comparison of \cite[Proposition~B.5]{RSVdB1} identifies
$(T(s)\otimes S)^{G_1}$ with $\mathsf I(L_s^{B_1})$.
Since $G/B=\mathbb P^1$ and hence $R^2\mathsf I=0$, its degree-one
exact sequence is
\[
 0\longrightarrow R^1\mathsf I(L_s^{B_1})
 \longrightarrow H^1(G_1,T(s)\otimes S)
 \longrightarrow \mathsf I H^1(B_1,L_s)
 \longrightarrow0.
\]
The Tate comparison in the same proposition, together with
$R^1\mathsf I\widehat H^0(B_1,L_s)=0$, gives
\[
 \widehat H^1(G_1,T(s)\otimes S)
 \cong\mathsf I\widehat H^1(B_1,L_s).
\]
Ordinary and Tate cohomology agree in degree one. The middle and
right-hand terms of the exact sequence therefore have equal dimensions
in each finite internal degree, forcing
$R^1\mathsf I(L_s^{B_1})=0$.
Consequently Euler induction computes the actual covariant character
from numerator $z^s$ and Tate numerators $H_{n,s}$ and $wH_{n,s}$.

\smallskip\noindent
\emph{Assembling the columns.}
For each covariant source $S\{s\}$, the diagonal syzygy multiplicities
are $H_{n,s}$. Subtracting their resolution characters from the induced
character is the same linear operation as the subtraction in
\eqref{gr:KS}. The matching weight and fusion inputs above therefore
give, in every output coordinate,
\[
 \mathcal G_{jk}(p)
 =\sum_{a=k+1}^{n-1}(-1)^{a-k-1}\binom n{b_a}A_n(p)e_{c_a}
 =A_n(p)t_{jk}.
\]
Subtracting this identity from
$A_n(p)f_{jk}=\mathcal G_{jk}(p)+\mathcal E_{jk}(p)$ proves the lemma.
The same comparison and Lemma~\ref{gr:boundary} include the endpoint
characteristic.
\end{proof}

For a vector in either residual space, write $[\,\cdot\,]_{\mathcal B^\pm}$
for its coordinates in the displayed basis. Define matrices by their
source columns:
\begin{equation}\label{gr:boundary-matrices}
 (H_n^+)_{\bullet,(j,k)}=
 \left[\frac{\mathcal E_{jk}+D\mathcal E_{jk}}2\right]_{\mathcal B^+}
 \qquad \text{and} \qquad
 (H_n^-)_{\bullet,(j,k)}=
 [\mathcal E_{jk}-D\mathcal E_{jk}]_{\mathcal B^-}.
\end{equation}
These are fixed rational coordinate operations. If $P^\pm$ is the
matrix of the corresponding basis in $\mathcal V_n$, its left inverse
$((P^\pm)^tP^\pm)^{-1}(P^\pm)^t$ extracts the coordinates.
For $H_n^+$ the covariant coordinates must be retained, since the
paired diagonal basis vectors share the same syzygy coordinates.

\begin{theorem}\label{gr:characteristic-polynomial}
For every $n\ge4$ and prime $p\ge\max\{n-2,3\}$, we have
\begin{equation}\label{gr:full-factorization}
 \det(tI-A_n(p))=
 \prod_{i=0}^{\lfloor n/2\rfloor-1}(t-p^{2n-3-2i})\,
 P_n^+(t,p)P_n^-(t,p),
 \qquad \text{with} \qquad P_n^\pm(t,p)=\det(tI-H_n^\pm(p)).
\end{equation}
The matrices $H_n^\pm(p)$ have sizes $d_+,d_-$ and entries in
$\mathbb Q[p]$ of degree at most $n-1$. Every entry is obtained from
the $n$-variable counts $C_{n,p}$ by fixed rational linear operations
with ranges depending only on $n$. Empty determinants are one.
\end{theorem}
\begin{proof}
Boundary transfer and duality identify \eqref{gr:boundary-matrices}
with the restrictions of $A_n(p)$ to $\mathcal N_n^\pm$.
The quotient has the Frobenius weights of
Theorem~\ref{gr:rr-eigenrows}. A basis adapted to $\mathcal N_n$
makes $A_n(p)$ block triangular, proving \eqref{gr:full-factorization}.
Each atom $E_{j;s,e}=C_{n,p}((s+e)p-j-2e)$ is a polynomial of degree
at most $n-1$ by Lemma~\ref{gr:binomial}. Every subsequent operation
in \eqref{gr:boundary-vector} and \eqref{gr:boundary-matrices} is
independent of $p$, proving the degree bound.
\end{proof}

Under the map that sums coordinates over duality orbits,
$\mathcal N_n^+$ identifies with the kernel of the matrix $W$ above.
Thus $H_n^+$ and the matrix $H$ of \eqref{gr:rr-residual} are related
by a constant rational change of basis. The boundary construction
computes its characteristic factor directly, and also supplies the
negative factor absent from \eqref{gr:rr-factorization}.
The extension data represented by $b$ in \eqref{gr:rr-residual} is
still needed for the signature equations.

For fixed $n$, every residual eigenvalue is $O_n(p^{n-1})$, since
this bounds the norms of $H_n^\pm(p)$. Hence each geometric weight
with exponent greater than $n-1$ is a simple root of the full
characteristic polynomial for all sufficiently large primes. These
are all $q$ geometric weights for odd $n$ and the largest $q-1$ for
even $n$. For $n=4$, the construction gives $H_4^+=(p(p^2+2)/3)$ and an empty
negative block, recovering the spectrum in \S\ref{gr:example}.

The boundary counts also have a fiber interpretation. A Koszul
calculation gives
\[
 \Tor_i^S(Q,S_+)\cong
 k[x_1,\ldots,x_n]/(x_1^p,\ldots,x_n^p)
 \otimes\bigwedge^i\langle y_1,\ldots,y_n\rangle^{[1]},
\]
with exterior generators in internal degree $p$. These are the Tor
modules of the derived intersection of the Frobenius fiber with the
incidence fiber underlying the resolutions. Their weight counts enter
the boundary reconstruction. No equivariant formality is required.

\section{Explicit signatures}\label{ex:examples}
The signature of $C_4$ was derived by hand in
Theorem~\ref{gr:example-thm}. For $n=5,6$ the same method gives the
following.

\begin{theorem}\label{ex:gr25}
For every prime $p\ge3$,
\[
 s_p(C_5)=\frac{N_5(p)}{168\,D_5(p)},
\]
where
\[
\begin{aligned}
 N_5(p)&=1552p^9+1552p^8+3104p^7+3677p^6+3947p^5\\
       &\qquad+3103p^4+2968p^3+1551p^2+843p+573,\\
 D_5(p)&=12p^9+12p^8+24p^7+29p^6+31p^5+25p^4+24p^3+13p^2+7p+5.
\end{aligned}
\]
In particular $s_3(C_5)=1226365/1595624$ and
$\lim_{p\to\infty}s_p(C_5)=97/126$.
\end{theorem}

\begin{theorem}\label{ex:gr26}
For every prime $p\ge5$,
\[
 s_p(C_6)=\frac{N_6(p)}{1512\,D_6(p)},\qquad
 N_6(p)=\sum_{j=0}^{30}n_jp^j,\quad D_6(p)=\sum_{j=0}^{30}d_jp^j,
\]
with the coefficients of Table~\ref{tab:gr26}. Moreover
\[
 s_p(C_6)=\frac{15067}{22680}-\frac{899}{48600\,p^4}+O(p^{-5}).
\]
\end{theorem}

\begin{table}[ht]\centering\small
\begin{tabular}{r|r|r@{\qquad}r|r|r}
$j$&$n_j$&$d_j$&$j$&$n_j$&$d_j$\\\hline
0&$-2203200$&$-2496$&16&$-79081424$&$-63896$\\
1&$-8786880$&$-9552$&17&$187168006$&$203366$\\
2&$-21552264$&$-23400$&18&$421151974$&$437262$\\
3&$-46964772$&$-50148$&19&$615200326$&$628962$\\
4&$-89401548$&$-94908$&20&$738355270$&$749550$\\
5&$-152948832$&$-160684$&21&$779041292$&$786528$\\
6&$-241372020$&$-251996$&22&$761408670$&$766134$\\
7&$-344233233$&$-356673$&23&$667306968$&$669240$\\
8&$-470178765$&$-484345$&24&$561043200$&$561600$\\
9&$-578750905$&$-592425$&25&$409481712$&$408960$\\
10&$-687982955$&$-700235$&26&$297903276$&$297180$\\
11&$-731863346$&$-740450$&27&$173571840$&$172800$\\
12&$-746891291$&$-750875$&28&$108482400$&$108000$\\
13&$-659348687$&$-657735$&29&$43392960$&$43200$\\
14&$-531128519$&$-523915$&30&$21696480$&$21600$\\
15&$-313172387$&$-301443$&&&\\
\end{tabular}
\caption{The coefficients of $N_6$ and $D_6$ in Theorem~\ref{ex:gr26}.}
\label{tab:gr26}
\end{table}

\begin{proof}[Proofs of Theorems~\ref{ex:gr25} and~\ref{ex:gr26}]
We use the boundary matrices of Section~\ref{gr:characteristic} to
reduce the calculations to three and six unknowns, respectively.
Write $d=2n-3$, $q=\lfloor n/2\rfloor$, and let $Q_n$ be the
matrix whose columns are the vectors of $\mathcal B^+$, ordered by
increasing $j$ and then increasing $k$. Choose the following vectors in the covariant span:
\[
 u_5=\frac56e_0+\frac1{12}e_1 \qquad \text{and} \qquad
 u_6=\frac{101}{144}e_0+\frac5{36}e_1+\frac1{144}e_2.
\]
Formula~\eqref{gr:rr-covariant} gives
$w_0(u_n)=1$ and $w_i(u_n)=0$ for $1\le i<q$.
Thus the normalized Perron vector has the form $v=u_n+Q_nz$.
Set
\[
 b_n=[A_n(p)u_n-p^du_n]_{\mathcal B^+},\qquad
 c_n=(Q_n)_{0,\bullet},\qquad \text{and} \qquad
 M_n=p^dI-H_n^+(p).
\]
The Riemann--Roch identities and duality place the vector defining
$b_n$ in $\mathcal N_n^+$. Boundary transfer gives
$A_n(p)Q_n=Q_nH_n^+(p)$, so the eigenvector equation becomes
\[
 M_nz=b_n \qquad \text{and} \qquad s_p(C_n)=(u_n)_0+c_nz.
\]
The Perron root is simple, hence $M_n$ is invertible at every prime
in the stated range.

The accompanying program constructs $H_n^+$ directly from the
$E$-counts in \eqref{gr:boundary-vector} and \eqref{gr:boundary-matrices}.
To obtain $b_n$, it uses only the first $q$ covariant-source columns
from \eqref{gr:SC}--\eqref{gr:SK}. It does not assemble the full
Frobenius matrix or the negative residual block. The resulting exact
linear algebra certifies the polynomial identities $\det M_n=g_nD_n$ and
\[\kappa_n\bigl((u_n)_0\det M_n+c_n\adj(M_n)b_n\bigr)=g_nN_n,\]
where $\kappa_5=168$, $\kappa_6=1512$, and
\[
 g_5=\frac{p^6(p-1)^3(p+1)(p^2+p+1)}{12}, \qquad 
 g_6=\frac{p^{14}(p-1)^5(p+1)^3(p^2+1)}{21600}.
\]
Here $N_n,D_n$ are exactly the polynomials in the statements.
Cramer's rule now gives $s_p(C_n)=N_n/(\kappa_nD_n)$.
Since $g_n$ and $\det M_n$ are nonzero at every prime in range, these
fractions are valid throughout that range. Expanding the verified
fraction for $n=6$ at $p=\infty$ gives the stated asymptotic.
All calculations use exact polynomial arithmetic; no approximation or interpolation is used.
\end{proof}

\begin{example}[The three-variable system for $C_5$]\label{ex:reduced-C5}
In the order $(y_{11},y_{21},y_{22})$, the residual matrix is
\[
 H_5^+(p)=\frac p{24}
 \begin{pmatrix}
 -p^3+14p^2+p+10 & -(p-1)(p^2-9p+2) & -(p^2-1)(p+2)\\
 2p(p^2-1) & 2p(p^2+11) & 2p(p^2-1)\\
 (p-3)(p-1)(13p+10) & (p-1)(13p^2-17p-6) & (p+1)(13p^2-7p+6)
 \end{pmatrix}.
\]
The right-hand side and free-coordinate row are
\[
 b_5=\frac{p(p^2-1)}{2520}
 \begin{pmatrix}
 -(2p^4+2p^2-5)\\0\\3(12p^4+12p^2+5)
 \end{pmatrix},\qquad c_5=(-10,-5,-5).
\]
Thus solving $(p^7I-H_5^+)z=b_5$ gives
$s_p(C_5)=5/6-10z_1-5z_2-5z_3$.
The analogous six-variable system for $C_6$, together with polynomial
certificates for both systems, is recorded in the ancillary package.
\end{example}

As independent consistency checks, the full-matrix program solves the full
normalized system over $\mathbb Q$ at $p=3,5,7$ for $C_4,C_5$ and at
$p=5,7$ for $C_6$, and verifies the Riemann--Roch eigenrow identities
for $4\le n\le8$. These
checks merely supplement the symbolic identities. They are not used to
interpolate the formulas.

\section{Concluding remarks}\label{sec:scope}
For $p\ge\max\{n-2,3\}$, the matrix $A_n(p)$ computes every Frobenius
multiplicity of $C_n$ by iteration, and hence the asymptotic
proportion of every summand. The characteristic bound is that of
\cite{RSVdB2}, so smaller characteristics require further arguments.
Theorem~\ref{gr:rr-eigenrows} reduces the signature to a linear system
in $\binom{n-2}{2}$ unknowns, but a direct scalar formula for $s_p(C_n)$
for general $n$ remains open. The boundary construction in
Theorem~\ref{gr:characteristic-polynomial} determines the complete
characteristic polynomial, but full diagonalizability and an explicit
description of all eigenspaces remain separate questions.

In the sequel we apply these results to the generic determinantal rings
$k[X_{a\times b}]/I_3(X)$. They are weight-zero parts of the rings
$C_{a+b}$ for a one-dimensional torus, and retaining that torus
character through the calculations of this paper descends them.

The reduced-signature program constructs the systems used in
Section~\ref{ex:examples} from the boundary counts and the first
$q$ covariant-source columns. The ancillary package includes their
exact matrices, right-hand sides, rational solutions, and cleared
polynomial certificates, together with a short certificate checker.
The full-matrix program implements the finite assembly formulas of
Section~\ref{gr:section} and Appendix~\ref{gr:appendix} as an independent
check. It also checks the counting functions $W_n$ and $H_{n,q}$ against direct
convolution and fusion multiplication for $4\le n\le10$, checks the
first free multiplicity \eqref{gr:a1} and Corollary~\ref{gr:limit} for
$4\le n\le10$, and verifies Theorem~\ref{gr:rr-eigenrows} for $4\le n\le8$.
Instructions, the exact matrices and the
normalized Perron vectors on duality orbits are included in the
ancillary package; see \texttt{anc/README.txt} for reproduction commands.
A separate check implements
\eqref{gr:boundary-vector} and \eqref{gr:boundary-matrices} literally.
For $4\le n\le7$, exact comparisons with the independently assembled full
matrices verify the boundary-transfer identities \eqref{gr:boundary-action}
as polynomial identities, and the entire characteristic polynomial at
several primes. These serve merely as implementation checks. 

\appendix
\section{The syzygy-column formulas}\label{gr:appendix}
This appendix records the finite formulas for the columns of a syzygy
source $K\{j,k\}$, used in \S\S\ref{gr:kk}--\ref{gr:ks}. It suffices to
consider $1\le k\le j\le n-3$, as sources with $k>j$ are reduced by 
duality \eqref{gr:duality}. Besides $W_n$ and $H_{n,q}$, the formulas use
the bounded-composition atom
\begin{equation}\label{gr:E}
 E_{j;q,e}=C_{n,p}\bigl((q+e)p-j-2e\bigr),\qquad e\in\{0,1\},
\end{equation}
which counts the tuples $t\in[0,p-1]^n$ of total degree $(q+e)p-j-2e$
carrying the surviving class described below.

Two facts organize the formulas. First, the syzygy targets require no
map-rank calculation. After fixing a tuple and a torus weight, the
spectral sequence \cite[(6.17)]{RSVdB2} has its $E_3$ page concentrated
in one column \cite[Lemma~7.8]{RSVdB2}. Its $E_1$ dimensions are fusion
coefficients, and the full complex has at most one surviving class, in
position $-q$, at tuple degree $(q+e)p-j-2e$ and row parity
$r\equiv q+e\pmod2$. Alternating dimensions to the left of position $-k$
give the generic term $B$ below, and adding or removing the unique class
gives the indicators in \eqref{gr:T} and \eqref{gr:O}.
\cite[Proposition~7.6]{RSVdB2} passes from the truncated complex to the
Borel-level syzygy, \cite[Corollary~8.4]{RSVdB2} induces to $G_1$,
\cite[Proposition~3.5(6)]{RSVdB2} identifies the stable syzygy of $M_r$
in cohomological degree $\ell$ with $K\{r,\ell\}$, and
\cite[Theorem~4.3]{RSVdB2} identifies the resulting ordinary-cohomology
counts with the syzygy-target multiplicities. Second, the formulas
\eqref{gr:P}--\eqref{gr:KS} for the tilt-free part are the expansion of
the Jordan-block identity \eqref{det:jordan-invariants}, as explained in
\S\ref{gr:ks}.

For $0\le m\le n$ with $m\ne j+1$, put
\[
 s_m=\begin{cases}-m,&m\le j,\\1-m,&m\ge j+2,\end{cases}
 \qquad \text{and} \qquad
 h_m=\begin{cases}j-m,&m\le j,\\m-j-2,&m\ge j+2.\end{cases}
\]
The generic term is
\begin{equation}\label{gr:B}
 B_{j,k}=\sum_{\substack{0\le m\le n,\ m\ne j+1\\-(n-1)\le s_m<-k}}
 (-1)^{-k-1-s_m}\binom nm H_{n,h_m}.
\end{equation}
In the weighted setting of the sequel this term depends on the parity of
the cohomological row, but here it does not, since
$H_{n,p-2-q}=H_{n,q}$. For $\ell=0,1$, the Tate numerators are
\begin{equation}\label{gr:T}
\begin{split}
 \mathcal T_{j,k;\ell,r}={}&[r=\ell]\,B_{j,k}
 +\sum_{e=0}^1\sum_{q=k+1}^{n-1-e}
   (-1)^{q-k}[r=\ell][\ell\equiv q+e]\,E_{j;q,e}\\
 &+\sum_{e=0}^1\sum_{q=k+1}^{n-1-e}
   [r=q+\ell-k-1][\ell-k-1\equiv e]\,E_{j;q,e},
\end{split}
\end{equation}
and for $1\le\ell\le n-3$, the ordinary-cohomology counts are
\begin{equation}\label{gr:O}
\begin{split}
 \mathcal O_{j,k;\ell,r}={}&[r=\ell]\,B_{j,k}
 +\sum_{e=0}^1\sum_{q=k+1}^{n-1-e}
   (-1)^{q-k}[r=\ell][\ell\equiv q+e]\,E_{j;q,e}\\
 &+[\ell\ge k]\sum_{e=0}^1\sum_{q=0}^{\min(k,n-1-e)}
   [r=q+\ell-k][\ell-k\equiv e]\,E_{j;q,e}\\
 &+[\ell\le k]\sum_{e=0}^1\sum_{q=k+1}^{n-1-e}
   [r=q+\ell-k-1][\ell-k-1\equiv e]\,E_{j;q,e}.
\end{split}
\end{equation}
Congruences are modulo $2$, and brackets denote indicators.

Set $P_{j,k;c}=0$ for $c<-n$, and otherwise
\begin{equation}\label{gr:P}
\begin{split}
 P_{j,k;c}={}&\sum_{m=0}^{n}(-1)^m\binom nm
   \bigl(\mathcal T_{j,k;1,c+m}-\mathcal T_{j,k;0,c+m}\bigr)\\
 &+\sum_{m=k+1}^{n}(-1)^{m-k-1}\binom nm
   \bigl(W_n(pc-j+m)-W_n(pc-j+m-2)\bigr).
\end{split}
\end{equation}
The reduced pre-induction character has coefficients
\begin{equation}\label{gr:Z}
 Z_{j,k;c}=\sum_{t=0}^{n-2}P_{j,k;c-2t}
   +\sum_{m=0}^{n}(-1)^m\binom nm\,\mathcal T_{j,k;0,c+m},
 \qquad \text{with} \qquad -(n-1)\le c\le n-3.
\end{equation}
The syzygy-target entries are
\begin{equation}\label{gr:KK}
 A_{K\{r,h\},K\{j,k\}}=\mathcal O_{j,k;h,r}.
\end{equation}
For $1\le r,h\le n-3$ and $h+1\le m\le n-1$, put
$(\eta_{r,m},\zeta_{r,m})=(r-m,m)$ if $m\le r$, and
$(\eta_{r,m},\zeta_{r,m})=(m-r-1,m+1)$ if $m>r$. For $0\le l\le n-3$ the
covariant-target entries are
\begin{equation}\label{gr:KS}
\begin{split}
 A_{S\{l\},K\{j,k\}}={}&Z_{j,k;l}-Z_{j,k;-l-2}\\
 &-\sum_{r,h=1}^{n-3}\sum_{m=h+1}^{n-1}
   (-1)^{m-h-1}[\eta_{r,m}=l]\binom n{\zeta_{r,m}}
   \mathcal O_{j,k;h,r}.
\end{split}
\end{equation}
Every expression is a finite sum, with ranges independent of $p$, of the
atoms $H_{n,q}$, $W_n$ and $E_{j;q,e}$ at arguments affine in $p$.

\subsection*{The boundary vector}
The construction in \S\ref{gr:characteristic} uses only the $E$-terms
of these formulas. Write $\mathcal T^E$ and $\mathcal O^E$ for
\eqref{gr:T} and \eqref{gr:O} with the generic term $B_{j,k}$ omitted.
For $\ell=0,1$, set
\[
 J_{j,k;\ell}(w)=\sum_r\mathcal T^E_{j,k;\ell,r}w^r
 \qquad \text{and} \qquad
 Z^E_{j,k}(w)=(1-w^{-1})^n
 \frac{J_{j,k;1}(w)-w^2J_{j,k;0}(w)}{1-w^2}.
\]
This is the specialization of \eqref{gr:P}--\eqref{gr:Z} retaining
only the $E$-atoms. Division is expanded at zero. All required
coefficients are obtained by replacing $(1-w^2)^{-1}$ with
$\sum_{v=0}^{n-2}w^{2v}$. The numerator has no exponent below $1-n$,
and the requested exponents range from $-(n-1)$ to $n-3$.
Using the tail classes of \eqref{gr:boundary-classes}, the full vector is
\begin{equation}\label{gr:boundary-vector}
 \mathcal E_{jk}(p)=
 \sum_{a=0}^{n-3}\bigl([w^a]Z^E_{j,k}-[w^{-a-2}]Z^E_{j,k}\bigr)e_a
 +\sum_{r,h=1}^{n-3}\mathcal O^E_{j,k;h,r}x_{rh}.
\end{equation}
The use of $x_{rh}$, for all output pairs including $r<h$, incorporates
exactly the covariant subtraction in \eqref{gr:KS}. Together with
\eqref{gr:E}, this specifies every input to
\eqref{gr:boundary-matrices} using only bounded-composition counts.

\section*{Acknowledgements}
Various AI models were used during the preparation of this work, being especially useful with finding the closed form expression of Corollary \ref{gr:limit}, finding a geometric rationale for the decomposition of the characteristic polynomial in Theorem \ref{gr:characteristic-polynomial}, and in building the computer checks of Section \ref{ex:examples}. They were also used extensively in drafting, proofreading, and editing this document. The author takes full responsibility for all mathematical content.

\vspace{.1in}

\end{document}